\documentclass[reqno,11pt, psfig]{amsart}
\usepackage{cases}
\usepackage{mathrsfs}
\usepackage[T1]{fontenc}
\usepackage{mathrsfs}
\usepackage{amsmath,latexsym,amssymb,amsfonts,amsbsy, amsthm}
\usepackage{bm}
\usepackage[usenames]{color}
\usepackage{xspace,colortbl}
\usepackage{epsfig}
\usepackage{graphicx}
\usepackage{subfigure}
\usepackage{amsmath,amsfonts,amsthm,amssymb,amscd}
\usepackage{color}
\usepackage[title,titletoc,toc]{appendix}
\usepackage{bigints}

\allowdisplaybreaks

\newcommand{\R }{\mathbb{R}}

\newcommand{\f}{\frac}

\newcommand{\LC}{\left(}
\newcommand{\RC}{\right)}

\newcommand{\Int}{\int_{\mathbb{R}}}

\newtheorem{remark}{Remark}[section]

\newtheorem{cor}{Corollary}[section]

\newtheorem{theorem}{Theorem}[section]

\begin{document}


\title[Shallow-water model with the Coriolis effect]
{A nonlocal shallow-water model arising from the full water waves with the Coriolis effect}
\author[Gui]{Guilong Gui}
\address{Guilong Gui\newline
School of Mathematics, Northwest University, Xi'an 710069, P. R. China}
\email{glgui@amss.ac.cn}
\author[Liu]
{Yue Liu}
\address{Yue Liu \newline
Department of Mathematics, University of Texas at Arlington, Arlington, TX 76019}
\email{yliu@uta.edu}
\author[Sun]
{Junwei Sun}
\address{Junwei Sun\newline
Department of Mathematics, University of Texas at Arlington, Arlington, TX 76019}
\email{junwei.sun@mavs.uta.edu}


\begin{abstract}
In the present study
a mathematical model of the equatorial water waves propagating mainly in one direction with the effect of Earth's rotation is derived by the formal asymptotic procedures in the equatorial zone. Such  a model equation is analogous to the Camassa-Holm approximation of the two-dimensional incompressible and  irrotational  Euler equations and has a formal bi-Hamiltonian structure.  Its solution corresponding to physically relevant initial perturbations is more accurate on a much longer time scale. It is shown that the deviation of the free surface can be determined by the horizontal velocity at a certain depth in the second-order approximation.  The  effects of the Coriolis force caused by the Earth rotation and  nonlocal higher nonlinearities on blow-up criteria and wave-breaking phenomena are also investigated. Our refined analysis is approached by applying the method of characteristics and conserved quantities to the  Riccati-type differential inequality.
\end{abstract}


\maketitle


\vskip 0.2cm


\renewcommand{\theequation}{\thesection.\arabic{equation}}
\setcounter{equation}{0}

\section{Introduction}

It is known  that many of the shallow water  models as approximations to the full Euler dynamics are only valid in the weakly nonlinear regime, for instance, the classical Korteweg-de Vries (KdV) equation \cite{KdV}
\[
u_t + u_x + \frac{3}{2}u u_x + \frac{1}{6} u_{xxx} = 0.
\]
However,  the more interesting physical phenomena, such as wave breaking, waves of maxima height \cite {AmTo, To}, require a transition to full {\it nonlinearity}. The KdV equation is a simple mathematical model for gravity waves in shallow water, but it fails to model fundamental physical phenomena such as the extreme wave of Stokes \cite{St} and does not include breaking waves (i.e.  wave profile  remains bounded while its slope becomes unbounded in finite time). The failure of weakly nonlinear shallow-water wave equations to model observed wave phenomena in nature is prime motivation in the search for alternative models for nonlinear shallow-water waves \cite {Ro, Wh}. The long-wave regime is usually characterized by presumptions of long wavelength $\lambda$  and small amplitude $a$ with the amplitude parameter $\varepsilon $ and the shallowness parameter $\mu $ respectively by
\begin{equation*} \label{parameter}
\varepsilon = \frac{a}{h_0} \ll 1, \qquad \mu = \frac{h_0^2}{\lambda^2} \ll 1.
\end{equation*}
It is well understood that the KdV model provides a good asymptotic approximations of unidirectional solutions of the irrotational two-dimensional water waves  problem  on the Boussinesq regime $ \mu \ll 1$, $\varepsilon = O(\mu) $ \cite{BCL05, C85}. To describe more accurately the motion of these unidirectional waves, it was shown in \cite{CL09} that the Camassa-Holm (CH) equation \cite{CH, FF} in the CH scaling, $  \mu \ll 1$, $\varepsilon = O(\sqrt{\mu}), $ could be valid higher order  approximations to the governing equation for full water waves in the long time scaling $ O(\frac{1}{\varepsilon})$. Like  the KdV, the CH equation is integrable and have solitons, while the  CH equation models breaking waves and has  peaked solitary waves \cite {CH, ce-1, Mc1}.  It is also found  that  the Euler equation has breaking waves \cite{BeOl} and a traveling-wave solution with the greatest height which has  a corner at its crest \cite{To}.

The Camassa-Holm equation inspired the search for various generalization of this equation with interesting properties and applications. Note  that all nonlinear terms in the CH equation is quadratic. It  is then of great interest to find those integrable equations with  higher-power nonlinear terms.

Analogous to the CH equation, our first  main aim of  the present paper is to formally derive a model equation with the Coriolis effect  from the incompressible and  irrotational two-dimensional shallow water in the equatorial region. This new model equation called the rotation-Camassa-Holm (R-CH) equation has a cubic and even quartic nonlinearities and  a formal Hamiltonian structure. More precisely, the motion of the fluid is described by the scalar equation in the form
\begin{equation}\label{R-CH-1}
\begin{split}
 u_t-\beta\mu  u_{xxt} + c u_x + 3\alpha\varepsilon uu_x - \beta_0\mu u_{xxx} &+ \omega_1 \varepsilon^2u^2u_x + \omega_2 \varepsilon^3u^3u_x   \\
&= \alpha\beta\varepsilon\mu( 2u_{x}u_{xx}+uu_{xxx}),
\end{split}
\end{equation}
 where the parameter $ \Omega $ is the constant rotational frequency due to the Coriolis effect. The other constants appearing in the equation  are defined by
$
c = \sqrt{1 + \Omega^2} - \Omega,  \;  \alpha \overset{\text{def}}{=} \frac{c^2}{1+c^2},  \, \beta_0  \overset{\text{def}}{=}\frac{c(c^4+6c^2-1)}{6(c^2+1)^2},  \, \beta \overset{\text{def}}{=}\frac{3c^4+8c^2-1}{6(c^2+1)^2},
$
$\omega_1 \overset{\text{def}}{=}\frac{-3c(c^2-1)(c^2-2)}{2(1+c^2)^3}, \, {\rm and}\, \omega_2 \overset{\text{def}}{=}\frac{(c^2-2)(c^2-1)^2(8c^2-1)}{2(1+c^2)^5}
$
satisfying  $c\to 1$, $\beta\to\f{5}{12}$, $\beta_0\to\f{1}{4}$, $\omega_1, \, \omega_2 \to 0$ and $\alpha\to\f{1}{2}$  when $\Omega\to 0$.
Denote $p_{\mu}(x)\overset{\text{def}}{=}\frac{1}{2\sqrt{\beta\mu}}e^{-\frac{|x|}{\sqrt{\beta\mu}}}$, $x\in \mathbb{R}$, then $(1-\beta\mu\partial_x^2)^{-1}f=p_\mu \ast f$ for all $f \in L^2(\mathbb{R})$ and $p_\mu \ast (u-\beta\mu u_{xx})=u$, where $\ast$ denotes
convolution with respect to the spatial variable $x$. With this notation, equation \eqref{R-CH-1} can also be equivalently rewritten as the following nonlocal form:
\begin{equation*}\label{nonlocal-form-1}
u_t +\frac{\beta_0}{\beta}u_x+\alpha\varepsilon u u_x+ p_\mu \ast \partial_x\bigg((c-\frac{\beta_0}{\beta}) u + \alpha \varepsilon u^2 +\frac{1}{2}\alpha\beta\varepsilon\mu u_x^2+ \frac{\omega_1}{3}\varepsilon^2 u^3+ \frac{\omega_2}{4}\varepsilon^3 u^4 \bigg) = 0,
\end{equation*}
or what is the same,
\begin{equation*}\label{nonlocal-form-2}
\begin{cases}
&u_t + \frac{\beta_0}{\beta}u_x+\alpha\varepsilon u u_x+ \partial_x P= 0,\\
&(1-\beta\mu\partial_x^2)P=(c-\frac{\beta_0}{\beta}) u + \alpha \varepsilon u^2 +\frac{1}{2}\alpha\beta\varepsilon\mu u_x^2+ \frac{\omega_1}{3}\varepsilon^2 u^3+ \frac{\omega_2}{4}\varepsilon^3 u^4.
\end{cases}
\end{equation*}
The solution $ u $ of \eqref{R-CH-1} represents the horizontal velocity field at height $ z_0$, and after the re-scaling, it is required that $ 0 \leq z_0 \leq 1, $ where
\begin{equation}\label{z-0-value}
z_0^2 = \frac {1}{2} - \frac{2}{3} \frac{1}{(c^2 + 1)} + \frac{4}{3} \frac{1}{(c^2 + 1)^2}.
\end{equation}
Since it is also natural to require that the constant $ \beta > 0, $ it must be the case
\[
0 \leq \Omega < \sqrt {\frac{1}{6} (1 + 2 \sqrt{19})} \approx 1.273,
\] and
\[
\frac{1}{\sqrt{2}} \leq z_0 < \sqrt{ \frac{61 - 2 \sqrt{19}} {54}} \approx 0.984.
\]
In particular,  when $ \Omega = 0, $  $ z_0= \frac{1}{\sqrt{2}} $ is corresponding  to the case of  classical CH equation.

The starting point of our derivation of the R-CH model in \eqref{R-CH-1} is the paper \cite{Jo1} where the classical CH equation  was derived.
The R-CH equation in \eqref{R-CH-1} is established by showing that  after a double asymptotic expansion with respect to $\varepsilon$ and $\mu$,  the free surface $\eta=\eta(\tau, \xi)$ under the field variable $ (\eta, \xi) $ defined in \eqref{notation-1}  in 2D Euler's dynamics \eqref{Euler-1} (see Section 2), is governed by the equation
\begin{equation*}\label{eta-eqns-1}
\begin{split}
   2(\Omega+c)\eta_{\tau} + 3c^2\eta\eta_{\xi} + \frac{c^2}{3}\mu\eta_{\xi\xi\xi} + A_1\varepsilon\eta^2\eta_{\xi} + A_2\varepsilon^2\eta^3\eta_{\xi} +A_{0}\varepsilon^3\eta^4\eta_{\xi}\\
  = \varepsilon\mu\Big[A_3\eta_{\xi}\eta_{\xi\xi} + A_4\eta\eta_{\xi\xi\xi}\Big]+O(\varepsilon^4, \mu^2),
\end{split}
\end{equation*}
where the constants $A_1 \overset{\text{def}}{=} \frac{3c^2(c^2-2)}{(c^2+1)^2}$, $ A_2 \overset{\text{def}}{=}  -\frac{c^2(2-c^2)(c^6-7c^4+5c^2-5)}{(c^2+1)^4}$, $A_3 \overset{\text{def}}{=}  \frac{-c^2(9c^4+16c^2-2)}{3(c^2+1)^2}$, $A_4 \overset{\text{def}}{=} \frac{-c^2(3c^4+8c^2-1)}{3(c^2+1)^2}$, $A_{0} \overset{\text{def}}{=} \frac{c^2(c^2-2)(3c^{10}+228c^8-540c^6-180c^4-13c^2+42)}{12(c^2+1)^6}$.
The free surface  $\eta $ with respect to the horizontal component of the velocity $u$ at $ z = z_0 $ under the CH regime $\varepsilon=O(\sqrt{\mu})$ is also given by
\begin{equation*}\label{surface-1}
\eta = \frac{1}{c} u + \gamma_1\varepsilon u^2 +\gamma_2\varepsilon^2 u^3+\gamma_3\varepsilon^3 u^4+\gamma_4 \varepsilon\mu u_{\xi\xi}+O(\varepsilon^4,\mu^2),
\end{equation*}
where the constants in the expression are given by $ \gamma_1=\frac{2-c^2}{2c^2(c^2+1)}$, $\gamma_2=\frac{(c^2-1)(c^2-2)(2c^2+1)}{2c^3(c^2+1)^3}$,
$\gamma_3=-\frac{(c^2-1)^2(c^2-2)(21c^4+16c^2+4)}{8c^4(c^2+1)^5}$, and $\gamma_4=\frac{z_0^2}{2c}-\frac{3c^2+1}{6c(c^2+1)}=\frac{-(3c^4+6c^2-5)}{12c(c^2+1)^2}$ (here the height parameter $z_0$ is determined by \eqref{z-0-value}).

Denote $m\overset{\text{def}}{=}(1-\beta\mu\partial_x^2)u$, one can  rewrite the above equation in terms of the evolution of the  momentum density $m$, namely,
\begin{equation}\label{R-CH-m}
\partial_t m +\alpha\varepsilon(um_x+2mu_x) +cu_{x} - \beta_0\mu u_{xxx} + \omega_1 \varepsilon^2u^2u_x + \omega_2 \varepsilon^3u^3u_x = 0.
\end{equation}
In the case that the Coriolis effect vanishes ($ \Omega = 0$), the coefficients in the higher-power nonlinearities $ \omega_1 = 0 $ and $ \omega_2 = 0.$ Using the scaling transformation
$u(t, x) \to \alpha \varepsilon u(\sqrt{\beta \mu}\,\,t,\sqrt{\beta \mu}\,\,x)$ and then the Galilean transformation $ u(t, x) \to u(t, x- \frac{3}{4}t) + \frac{1}{4}, $ the  R-CH equation \eqref{R-CH-m} is then reduced to the classical CH equation
\begin{equation*}\label{CH-1}
u_t - u_{xxt} + 3 uu_x = 2 u_x u_{xx} + u u_{xxx}.
\end{equation*}
On the other hand, if we take formally $\beta=0$ and $\omega_2=0$ in \eqref{R-CH-m}, then we get  the following integrable Gardner equation \cite{Gard68}
\begin{equation*}
u_t + c u_x + 3\alpha\varepsilon uu_x - \beta_0\mu u_{xxx} + \omega_1 \varepsilon^2u^2u_x  = 0.
\end{equation*}
Note that the R-CH equation \eqref{R-CH-m}  has the following three conserved quantities
\[
I(u) =\Int u\, dx, \quad  E(u)=\frac{1}{2}\Int u^2+\beta\mu u_x^2\,dx,
\] and
\[
 F(u)=\frac{1}{2}\Int cu^2+\alpha\varepsilon u^3+\beta_0\mu u_x^2+\frac{\omega_1 \varepsilon^2}{6}u^4 + \frac{\omega_2 \varepsilon^3}{10}u^5 +\alpha\beta\varepsilon\mu uu^2_x\,dx.
\]
Define  that
\begin{equation*}
\begin{split}
B_1 &\overset{\text{def}}{=} \partial_x(1-\beta\mu\partial^2_x),   \qquad {\rm and} \\
B_2 &\overset{\text{def}}{=} \partial_x((\alpha\varepsilon m+\frac{c}{2})\cdot)+(\alpha\varepsilon m+\frac{c}{2})\partial_x-\beta_0\mu\partial_x^3+\frac{2}{3}\omega_1\varepsilon^2\partial_x(u\partial_x^{-1}(u\partial_x\cdot)) \\
&\qquad\qquad\qquad \qquad\qquad\qquad\qquad \qquad+ \frac{5}{8}\omega_2\varepsilon^3\partial_x(u^{\frac{3}{2}} \partial_x^{-1}(u^{\frac{3}{2}}\partial_x\cdot)).
\end{split}
\end{equation*}
A simple calculation then reveals that the  R-CH equation \eqref{R-CH-1} can be written as
\begin{equation*}
m_t=-B_1\frac{\delta F}{\delta m} = -B_2\frac{\delta E}{\delta m},
\end{equation*}
where $B_1$ and $B_2$ are two skew-symmetric differential  operators.

The class of evolution equations \eqref{R-CH-1} are all formally models for  small amplitude, long waves on the surface  of water over a flat bottom.   It is our expectation that these equations approximate solutions of the full water-wave problem with the Coriolis effect for an ideal fluid with an error that is of order $ O(\mu^2t) $ over a CH time scale at least of order  $ O(\varepsilon^{-1}). $ Rigorous justification to this effect is available in \cite {CGL16} (see also \cite{CL09} for the case  without the Coriolis effect).

It is also found  that the consideration of the Coriolis effect gives rise to  a higher power nonlinear term into the R-CH model, which has interesting implications for the fluid motion, particular in the relation to the wave breaking phenomena and the permanent waves. On the other hand, it is also our goal in the present paper to investigate  from this model how the Coriolis forcing due to the Earth rotation with the higher power nonlinearities affects the wave breaking phenomena and what conditions can ensure  the occurrence of  the wave-breaking phenomena or permanent  waves.

The dynamics of the blow-up quantity along the characteristics in the R-CH equation actually involves the interaction among three parts: a local nonlinearity, a nonlocal term, and a term stemming from the weak Coriolis forcing. It is observed that  the nonlocal (smoothing) effect can help maintain the regularity while waves propagate and hence prevent them from blowing up, even when dispersion is weak or absent. See, for example, the Benjamin-Bona-Mahoney (BBM) equation \cite{BBM}.  As the local  nonlinearity becomes stronger and dominates over the dispersion and nonlocal effects singularities may occur in the sense of {\it wave-breaking}. Examples can be found in the Whitham equation \cite {ce-1, Wh}, Camassa-Holm (CH) equation \cite{CH,CL09, FF}. It is also found that the Coriolis effect will spread out waves and make them decay in time, delaying the onset of wave-breaking. Understanding the wave-breaking mechanism such as when a singularity can form and what the nature of it is not only presents fundamental importance from mathematical point of view but also is of great physical interest, since it would help provide a key-mechanism for localizing energy in conservative systems by forming one or several small-scale spots. For instance, in fluid dynamics, the possible phenomenon of finite time breakdown for the incompressible Euler equations signifies the onset of turbulence in high Reynolds number flows.

The R-CH equation with a nonlocal structure can be reformulated in a weak form of nonlinear nonlocal transport type. From the transport theory, the blow-up criteria assert that singularities are caused by the focusing of characteristics, which involve the information on the gradient $u_x$. The dynamics of the wave-breaking quantity along the characteristics  is established by the Riccati-type differential inequality.   The argument  is then approached by a refined analysis on  evolution of the solution $ u $ and its gradient $ u_x $.  Recently Brandolese and Cortez \cite{BrCo1} introduced a new type of blow-up criteria in the study of the classical CH equation. It is shown how local structure of the solution affects the blow-ups. Their argument relies heavily on the fact that the convolution terms are quadratic and positively definite. As for the R-CH equation, the convolution contains cubic even quartic nonlinearities which do not have a lower bound in terms of the local terms. Hence the higher-power nonlinearities in the equation makes it difficult to obtain a purely local condition on  the initial data can generate finite-time wave-breaking. In our case, the blow-up can be deduced  by the interplay between $u$ and $u_x$. More precisely, this motivates us to carry out a refined analysis of the characteristic dynamics of $M = u - u_x + c_1 $ and $N = u + u_x + c_2$.
The estimates of $M$ and $N$ can be closed in the form of
\begin{equation*}\label{estimates MN}
M'(t)\geq - cMN + \mathcal{N}_1, \quad N'(t) \leq c MN + \mathcal{N}_2,
\end{equation*}
where the  nonlocal terms $\mathcal{N}_i \;(i=1,2)$ can be bounded in terms of certain  order conservation laws. From these Riccati-type differential inequalities the monotonicity of $M$ and $N$ can be established, and hence the finite-time wave-breaking follows.

The present contribution proceeds in the following. In the next section, the R-CH model equation is  formally derived from the incompressible and irrotational full water wave equations with the Coriolis effect considered, which is  an asymptotic model in the CH regime to the $f$-plane geophysical governing equations in the equatorial region.  Sections \ref{local} is  devoted to the local well-posedness and blow-up criteria. In the last section, Section \ref{breaking}, the wave-breaking criteria are  established in Theorem \ref{thm-wavebreak-crt} and  the breakdown mechanisms are set up in Theorem \ref{Blow-up}.

\smallskip

\noindent{\bf Notation.} In the sequel, we denote by $\ast$ the convolution. For
$1\leq p<\infty$, the norms in the Lebesgue space $L^p(\R)$ is
$\|f\|_{p}=\Big(\int_{\R}|f(x)|^pdx\Big)^{\frac1p}$, the space
$L^{\infty}(\R)$ consists of all essentially bounded, Lebesgue
measurable functions $f$ equipped with the norm $\displaystyle
\|f\|_{\infty}=\inf_{\mu(e)=0}\sup_{x\in \R\setminus e}|f(x)|$.
 For  a function $f$ in the classical Sobolev
spaces $H^s(\R)\;(s\geq0)$ the norm is denoted by $ \|f\|_{H^s} $. We denote $p(x) = {1\over2} e^{-|x|}$ the fundamental solution of $ 1
- \partial^2_x $ on $\R$, and define the two convolution operators
$p_+, \;p_-$ as
\begin{equation*}\label{convo}
\begin{split}
& p_+ \ast f (x) = {e^{-x} \over 2} \int^x_{-\infty} e^y f(y) dy\\
& p_- \ast f(x) = {e^{x}\over 2} \int^\infty_{x} e^{-y} f(y) dy.
\end{split}
\end{equation*}
Then we have the relations $ \displaystyle
p = p_+ + p_-, \quad p_x = p_- - p_+. $

\renewcommand{\theequation}{\thesection.\arabic{equation}}
\setcounter{equation}{0}

\section{Derivation of the R-CH model} \label{derivation}
The formal derivation of the  Camassa-Holm model equation with the Coriolis effect in the equatorial region is the topic of the present section. Attention is given here is the so-called long-wave limit. in this setting, it is assumed that water flows are incompressible and inviscid with a constant density $\rho$ and no surface tension,  and the interface between the air and the water is a free surface. Then such a motion of water flow occupying a domain $\Omega_t$ in $\mathbb{R}^3$ under the influence of the gravity $ g $   and the Coriolis force due to the Earth's rotation
 can be described by the  Euler equations \cite{GSR07}, {\it viz.}
\begin{equation*} \label{R-Euler}
\begin{cases}
 &\vec{u}_t+\LC \vec{u}\cdot\nabla \RC\vec{u} + 2 \vec {\Omega} \times \vec{u} =-{1\over\rho}\nabla P +\vec{g},\quad x\in \Omega_t,\\
 &\nabla\cdot \vec{u}=0, \quad x\in \Omega_t,\\
 &\vec{u}|_{t=0}=\vec{u_0}, \quad x\in \Omega_0,
 \end{cases}
\end{equation*}
where $ \vec u = (u, v, w )^T $ is  the fluid velocity, $ P(t,x,y,z)$ is the pressure in the fluid, $ \vec{g} = ( 0, 0, -g )^T$ with $g \approx 9.8 m/s^2$ the constant gravitational acceleration at the Earth's surface, and $\vec \Omega  =  ( 0, \,  \Omega_0 \cos \phi, \,  \Omega_0 \sin \phi)^T$, with the rotational frequency  $\Omega_0 \approx 73\cdot 10^{-6}$rad/s and the local latitude $ \phi$,  is the angular velocity vector which is directed along the axis of rotation of the rotating reference frame. We adopt a rotating framework with the origin located at a point on the Earth's surface, with the $x$-axis chosen horizontally due east, the $y$-axis horizontally due north and the $z$-axis upward. We consider here waves at the surface of water with a flat bed, and assume that $\Omega_t=\{(x, y, z): 0<z<h_0+\eta(t, x, y)\}$, where $h_0$ is the typical depth of the water and $\eta(t, x, y)$ measures the deviation from the average level. Under the $f$-plane approximation $( \sin \phi \approx 0, \;  \phi \ll 1)$, the motion of inviscid irrotational fluid near the Equator in the region $0 < z < h_0 + \eta(t,x,y)$ with a constant density $\rho$ is described by the Euler equations \cite{Con12,GSR07} in the form
\begin{equation*} \label{f-plane}
\begin{cases}
u_t + uu_x+ vu_y + wu_z + 2\Omega_0 w = -\frac{1}{\rho}P_x, \\
v_t + uv_x+ vv_y + wv_z = -\frac{1}{\rho}P_y, \\
w_t + uw_x + vw_y + ww_z - 2\Omega_0 u = -\frac{1}{\rho}P_z - g,
\end{cases}
\end{equation*}
the incompressibility of the fluid,
\begin{equation*}\label{incom-1}
u_x + v_y + w_z = 0,
\end{equation*}
and the irrotational condition,
\begin{equation*}\label{irrot-1}
(w_y-v_z, u_z-w_x, v_x-u_y)^T = (0,0,0)^T.
\end{equation*}
The pressure is written as
\begin{equation*}
P(t, x,z) = P_a + \rho g(h_0 - z) + p(t, x, y, z),
\end{equation*}
where $P_a$ is the constant atmosphere pressure, and $p$ is a pressure variable measuring the hydrostatic pressure distribution.

The dynamic condition posed on the surface $z = h_0 + \eta$ yields $P = P_a$. Then there appears that
\begin{equation*}\label{perssure-1}
p = \rho g \eta.
\end{equation*}
Meanwhile, the kinematic condition on the surface is given by
\begin{equation*}\label{KC-1}
w = \eta_t + u\eta_x + v\eta_y, \quad \mbox{when} \quad z = h_0 + \eta(t, x, y).
\end{equation*}
Finally, we pose "no-flow" condition at the flat bottom $z = 0$, that is,
\begin{equation*}\label{bottom-1}
w|_{z=0} = 0.
\end{equation*}

Consider the two-dimensional flows, moving in the zonal direction along the equator independent of the $y$-coordinate, in other words, $v \equiv 0$ throughout the flow, the irrotational condition will be simplified as $u_z-w_x=0$. According to the magnitude of the physical quantities, we introduce dimensionless quantities as follows
\begin{equation*}
x \rightarrow \lambda x,\quad z \rightarrow h_0 z,\quad \eta \rightarrow a \eta,\quad t \rightarrow \frac{\lambda}{\sqrt{gh_0}}t,
\end{equation*}
which implies
\begin{equation*}
u \rightarrow \sqrt{gh_0}u,\quad w \rightarrow \sqrt{\mu gh_0} w,\quad p \rightarrow \rho g h_0 p.
\end{equation*}
And under the influence of the Earth rotation, we introduce
\begin{equation*}\label{rescall-1}
\Omega = \sqrt{{h_0}/{g}} \, \Omega_0.
\end{equation*}
Furthermore, considering whenever $\varepsilon \rightarrow 0$,
\begin{equation*}
u \rightarrow 0,\quad w \rightarrow 0,\quad p \rightarrow 0,
\end{equation*}
that is, $u, w$ and $p$ are proportional to the wave amplitude so that
we require a scaling
\begin{equation*}\label{rescall-2}
u \rightarrow \varepsilon u,\quad w \rightarrow \varepsilon w,\quad p \rightarrow \varepsilon p.
\end{equation*}
Therefore the governing equations become
\begin{equation}\label{governing}
\begin{cases}
u_t + \varepsilon(uu_x+wu_z)+2\Omega w = - p_x &\text{in}\quad 0 < z < 1+\varepsilon\eta(t,x),  \\
\mu \{w_t + \varepsilon (uw_x + ww_z)\} - 2\Omega u = -p_z &\text{in}\quad 0 < z < 1+\varepsilon\eta(t,x), \\
u_x + w_z = 0 &\text{in}\quad 0 < z < 1+\varepsilon\eta(t,x),\\
u_z - \mu w_x = 0 &\text{in}\quad 0 < z < 1+\varepsilon\eta(t,x), \\
p = \eta &\text{on}\quad z = 1 + \varepsilon\eta(t,x),\\
w = \eta_t + \varepsilon u \eta_x &\text{on}\quad z = 1 + \varepsilon\eta(t,x),\\
w = 0 & \text{on}\quad z = 0.
\end{cases}
\end{equation}

To derive the R-CH equation for shallow water waves, we first introduce a suitable scale and a double asymptotic expansion to get equations in groups with respect to $\varepsilon$ and $\mu$ independent on each other, where $\varepsilon,\, \mu \ll 1$.

Let $c$ be the group speed of water waves. We can apply a suitable far field variable together with a propagation problem \cite{Jo1, Jo2}
\begin{equation} \label{notation-1}
\xi = \varepsilon^{1/2}(x-ct),\quad \tau = \varepsilon^{3/2}t,
\end{equation}
which implies, for consistency from the equation of mass conservation, that we also transform
\begin{equation*}
w = \sqrt{\varepsilon} \,W.
\end{equation*}
Then the governing equations \eqref{governing} become
\begin{equation}\label{Euler-1}
\begin{cases}
- c u_{\xi} + \varepsilon (u_\tau + uu_\xi + Wu_z) + 2\Omega W = - p_\xi \quad &  \text{in}\quad 0 < z < 1 + \varepsilon \eta,\\
\varepsilon\mu \{- c W_\xi + \varepsilon (W_\tau + u W_\xi + WW_z)\} - 2\Omega u = - p_z \quad &  \text{in}\quad 0 < z < 1 + \varepsilon \eta,\\
u_\xi + W_z = 0 \quad & \text{in}\quad 0 < z < 1 + \varepsilon \eta,\\
u_z - \varepsilon\mu W_\xi = 0 \quad & \text{in}\quad 0 < z < 1 + \varepsilon \eta,\\
p = \eta \quad & \text{on}\quad z = 1+ \varepsilon \eta,\\
W = - c \eta_\xi + \varepsilon (\eta_\tau + u \eta_\xi) \quad & \text{on}\quad z = 1+ \varepsilon \eta,\\
W = 0 \quad & \text{on} \quad z = 0.
\end{cases}
\end{equation}
A double asymptotic expansion is introduced to seek a solution of the system \eqref{Euler-1},
\begin{equation*}
q \sim \sum_{n=0}^{\infty} \sum_{m=0}^{\infty}\varepsilon^n \mu^m q_{nm}
\end{equation*}
as $\varepsilon \rightarrow 0, \mu \rightarrow 0$, where $q$ will be taken the scale functions $u, \,W, \,p$ and $\eta$, and all the functions $q_{nm}$ satisfiy the far field conditions $q_{nm} \rightarrow 0$ as $|\xi|\rightarrow \infty$ for every $n, \,m=0, 1, 2, 3, ...$.

Substituting the asymptotic expansions of $u, \,W, \,p,\,\eta$ into \eqref{Euler-1}, we check all the coefficients of the order $O(\varepsilon^i\mu^j)$ ($i, \, j=0, 1, 2, 3, ...$).

From the order $O(\varepsilon^0 \mu^0)$ terms of \eqref{Euler-1} we obtain from the Taylor expansion
\begin{equation}\label{taylor-1}
f(z)=f(1)+\sum_{n=1}^{\infty} \frac{(z-1)^n}{n!}f^{(n)}(1)
\end{equation}
that
\begin{equation}\label{equation-00}
\begin{cases}
-c u_{00,\xi} + 2\Omega W_{00} = - p_{00,\xi} &\text{in}\quad 0 < z < 1,\\
2\Omega u_{00} = p_{00,z} &\text{in}\quad 0 < z < 1,\\
u_{00,\xi} + W_{00,z} = 0 &\text{in}\quad 0 < z < 1,\\
u_{00,z} = 0 &\text{in}\quad 0 < z < 1,\\
p_{00} = \eta_{00}, \quad W_{00} = - c \eta_{00,\xi} & \text{on} \quad z = 1,\\
W_{00} = 0 & \text{on} \quad z = 0.
\end{cases}
\end{equation}
To solve the system \eqref{equation-00}, we first obtain from the fourth equation in \eqref{equation-00} that $u_{00}$ is independent of $z$, that is,
$u_{00} = u_{00}(\tau, \xi)$.
 Thanks to the third equation in \eqref{equation-00} and the boundary condition of $W$ on $z=0$, we get
\begin{equation} \label{w00-1}
W_{00} =W_{00}|_{z = 0} + \int_0^z W_{00,z'} dz' = -\int_0^z u_{00,\xi}\, dz'= - z u_{00,\xi},
\end{equation}
which along with the boundary condition of $W$ on $z=1$ implies
$u_{00,\xi}(\tau, \xi) = c\eta_{00,\xi}(\tau, \xi)$.
Therefore, we have
\begin{equation}\label{w-00}
u_{00}(\tau, \xi) = c\eta_{00}(\tau, \xi), \quad W_{00} = -cz\eta_{00,\xi},
\end{equation}
here use has been made of the far field conditions $u_{00}, \, \eta_{00} \rightarrow 0$ as $|\xi| \rightarrow \infty$.

On the other hand, from the second equation in \eqref{equation-00}, there appears that
\begin{equation}\label{p-00-1}
p_{00}= p_{00}|_{z = 1} + \int_1^z p_{00,z'} \,dz'=\eta_{00}+2\Omega \int_1^z u_{00} \,dz'=\eta_{00}+2\Omega (z-1) u_{00},
\end{equation}
which along with $u_{00,\xi} = c\eta_{00,\xi}$ implies
\begin{equation}\label{p-00-2}
p_{00, \xi}=\big(\frac{1}{c}+2\Omega (z-1)\big) u_{00, \xi},
\end{equation}
Combining \eqref{p-00-2} with \eqref{w00-1} and the first equation in \eqref{equation-00} gives rise to
$(c^2 + 2\Omega c - 1) u_{00, \xi} = 0$,
which follows that
\begin{equation}\label{c-1-2}
c^2 + 2\Omega c - 1= 0,
\end{equation}
if we assume that $u_{00}$ is an non-trivial velocity. Therefore, when consider the waves move towards to the right side, we may obtain
\begin{equation}\label{c-1-3}
c = \sqrt{1 + \Omega^2} - \Omega.
\end{equation}

Similarly, vanishing the orders $O(\varepsilon^0 \mu^1)$, $O(\varepsilon^2 \mu^0)$, $O(\varepsilon^1 \mu^1)$, $O(\varepsilon^3 \mu^0)$, $O(\varepsilon^4 \mu^0)$, and $O(\varepsilon^2 \mu^1)$ terms of \eqref{Euler-1} respectively, we may obtain
\begin{equation}\label{uwp-01-1}
\begin{split}
&u_{01} = c\eta_{01}=c\eta_{01}(\tau, \xi),\\
&u_{20}=u_{20}(\tau, \xi)= c \eta_{20} - 2(c+c_1)\eta_{00}\eta_{10}-\frac{2c_1-3\Omega}{3(c+\Omega)}(c+c_1)\eta_{00}^3,\\
&u_{11} =u_{11}(\tau, \xi)= \left(\frac{c}{6} - \frac{2c_1}{9} -\frac{c}{2}z^2 \right) \eta_{00,\xi\xi} + c \eta_{11} - 2 (c+c_1) \eta_{00}\eta_{01},\\
&u_{30}=u_{30}(\tau, \xi)= c \eta_{30} -  2(c+c_1)(\eta_{00}\eta_{20}) -(c+c_1)(\eta_{10}^2)-\frac{2c_1-3\Omega}{\Omega+c}(c+c_1)(\eta_{00}^2\eta_{10})\\
&\qquad -\frac{(64cc_1+24c_1^2+45c^2+24\Omega^2-3)}{24(c+\Omega)^2}(c+c_1)(\eta_{00}^4),
\end{split}\end{equation}
and
\begin{equation}\label{eta-30-eqn}
\begin{split}
&2(c+\Omega)\eta_{30,\tau} +3c^2(\eta_{00}\eta_{30}+\eta_{10}\eta_{20})_{\xi} -2(3c+2c_1)(c+c_1)(\eta_{00}^2\eta_{20}+\eta_{00}\eta_{10}^2)_{\xi}\\
&\quad-\frac{(64cc_1+24c_1^2+45c^2-15)}{3(c+\Omega)}(c+c_1)(\eta_{00}^3\eta_{10})_{\xi}-B_2(\eta_{00}^5)_{\xi}=0,\\
&2(\Omega + c)\eta_{11,\tau} + 3c^2(\eta_{00}\eta_{11}+\eta_{10}\eta_{01})_\xi-2(c+c_1)(3c+2c_1)(\eta_{00}^2\eta_{01})_\xi+\frac{c^2}{3}\eta_{10,\xi\xi\xi}\\
&-\left(\frac{c^2}{6}+\frac{10c c_1}{9}+\frac{2 c_1^2}{9}\right)(\eta_{00,\xi}^2)_{\xi}-\left(\frac{c^2}{3}+\frac{20 c c_1}{9}+\frac{8 c_1^2}{9}\right)(\eta_{00}\eta_{00,\xi\xi})_{\xi}=0.
\end{split}\end{equation}
with
\begin{equation}\label{c-0-0}
c_1 \overset{\text{def}}{=} -\frac{3c^2}{4(\Omega + c)}=-\frac{3 c^3}{2 (c^2 + 1)},
\end{equation}
\begin{equation*}\begin{split}
B_1\overset{\text{def}}{=} &\frac{(c+c_1)^2(82cc_1+36c_1^2+45c^2-18\Omega c_1-27\Omega c-15)}{3(\Omega+c)^2}\\
&+\frac{c_1(c+c_1)(64cc_1+24c_1^2+45c^2+24\Omega^2-3)}{3(\Omega+c)^2},
\end{split}\end{equation*}
and
\begin{equation*}\begin{split}
B_2&\overset{\text{def}}{=} \frac{1}{5}B_1-\frac{(c+c_1)^2(2c_1-3\Omega)}{3(\Omega+c)}+\frac{2c(c+c_1)(64cc_1+24c_1^2+45c^2+24\Omega^2-3)}{12(\Omega+c)^2}\\
&=\frac{c^2(2-c^2)(3c^{10}+228c^8-540c^6-180c^4-13c^2+42)}{60(c^2+1)^6}.
\end{split}\end{equation*}
More details should be found in Appendix A.

Taking $\eta := \eta_{00} + \varepsilon \eta_{10} + \varepsilon^2 \eta_{20}+ \varepsilon^3 \eta_{30}+ \mu \eta_{01} + \varepsilon\mu\eta_{11}+O(\varepsilon^4,\mu^2)$.
Multiplying the equations \eqref{A-eta-00-eqn}, \eqref{A-eta-10-eqn}, \eqref{A-eta-01-eqn}, \eqref{A-eta-20-eqn}, \eqref{A-eta-30-eqn}, and \eqref{A-eta-11-eqn} by $1$, $\varepsilon$, $\mu$, $\varepsilon^2$, $\varepsilon^3$, and $\varepsilon \mu$, respectively, and then summating the results, we get the equation of $\eta$ up to the order $O(\varepsilon^4,\mu^2)$ that
\begin{equation}\label{etaepsilon3}
\begin{split}
&2(\Omega + c) \eta_\tau + 3 c^2 \eta \eta_\xi + \frac{c^2}{3} \mu \eta_{\xi\xi\xi} + \varepsilon A_1 \eta^2 \eta_\xi+\varepsilon^2 A_2 \eta^3 \eta_{\xi}+A_0\varepsilon^3\eta^4 \eta_{\xi} \\
&= \varepsilon \mu\bigg(A_3 \eta_{\xi}\eta_{\xi\xi} +A_4 \eta\eta_{\xi\xi\xi}\bigg) + O(\varepsilon^4,\mu^2),
\end{split}
\end{equation}
where $c_1 =-\frac{3 c^3}{2 (c^2 + 1)}$ is defined in \eqref{c-0-0},
$ A_1 \overset{\text{def}}{=}  - 2(3c+2c_1)(c+c_1)= \frac{3c^2(c^2-2)}{(c^2+1)^2}$,
 $ A_2 \overset{\text{def}}{=}  -\frac{(64cc_1+24c_1^2+45c^2-15)}{3(c+\Omega)}(c+c_1)= -\frac{c^2(2-c^2)(c^6-7c^4+5c^2-5)}{(c^2+1)^4}$,
$ A_3  \overset{\text{def}}{=}  \frac{2c^2}{3}+\frac{40c c_1}{9}+\frac{4 c_1^2}{3}= \frac{-c^2(9c^4+16c^2-2)}{3(c^2+1)^2}$,
$ A_4  \overset{\text{def}}{=}   \frac{c^2}{3}+\frac{20 c c_1}{9}+\frac{8 c_1^2}{9}=\frac{-c^2(3c^4+8c^2-1)}{3(c^2+1)^2}$, $A_{0} \overset{\text{def}}{=} \frac{c^2(c^2-2)(3c^{10}+228c^8-540c^6-180c^4-13c^2+42)}{12(c^2+1)^6}$.

On the other hand, notice that
$u_{00} = c \eta_{00}$, $u_{10} = c\eta_{10} - (c_1 + c)\eta^2_{00}$, $u_{01} = c \eta_{01}$, $u_{11} = c \eta_{11} - 2(c_1 + c)\eta_{00}\eta_{01} + \left( \frac{c}{6}-\frac{2c_1}{9}-\frac{cz^2}{2}\right)\eta_{00,\xi\xi}$,
$u_{20}= c \eta_{20} - 2(c+c_1)(\eta_{00}\eta_{10})-\frac{2c_1-3\Omega}{3(c+\Omega)}(c+c_1)(\eta_{00}^3)$,
 and
\begin{equation*}
\begin{split}
u_{30}= c \eta_{30} -  2(c+c_1)(\eta_{00}\eta_{20}) &-(c+c_1)(\eta_{10}^2)-\frac{2c_1-3\Omega}{\Omega+c}(c+c_1)(\eta_{00}^2\eta_{10})\\
&-\frac{(64cc_1+24c_1^2+45c^2+24\Omega^2-3)}{24(c+\Omega)^2}(c+c_1)(\eta_{00}^4),
\end{split}
\end{equation*}
we obtain
\begin{equation*}\label{u-equation-12}
\begin{split}
&\eta_{00}=\frac{1}{c}u_{00},\, \eta_{10}=\frac{1}{c}u_{10}+\gamma_1 u^2_{00}, \, \eta_{01}=\frac{1}{c}u_{01},\, \eta_{20} =\frac{1}{c}u_{20}+ 2\gamma_1 u_{00}u_{10}+\gamma_2 u_{00}^3,\\
& \eta_{30}= \frac{1}{c}u_{30} +\gamma_1 u_{10}^2+ 2 \gamma_1 u_{00}u_{20}+ 3 \gamma_2 u_{00}^2u_{10}+\gamma_3 u_{00}^4,\\
& \eta_{11} = \frac{1}{c}u_{11} + 2 \gamma_1 u_{00}u_{01} +\gamma_4 u_{00,\xi\xi},
\end{split}
\end{equation*}
where  $\gamma_1 \overset{\text{def}}{=} \frac{c_1 + c}{c^3}$, $\gamma_2 \overset{\text{def}}{=}\frac{2(c+c_1)^2}{c^5}+\frac{(2c_1-3\Omega)(c+c_1)}{3c^4(c+\Omega)}$,
$\gamma_3\overset{\text{def}}{=}\frac{5(c+c_1)^3}{c^7}+\frac{5(2c_1-3\Omega)(c+c_1)^2}{3c^6(c+\Omega)}
+\frac{(64cc_1+24c_1^2+45c^2+24\Omega^2-3)}{24c^5(c+\Omega)^2}(c+c_1)$,
$\gamma_4 \overset{\text{def}}{=} -\left( \frac{1}{6c}-\frac{2c_1}{9c^2}-\frac{z^2}{2c}\right)$,
or it is the same,
\begin{equation}\label{gamma-defi-2}
\begin{split}
&\gamma_1=\frac{2-c^2}{2c^2(c^2+1)}, \quad \gamma_2 =\frac{(c^2-1)(c^2-2)(2c^2+1)}{2c^3(c^2+1)^3},\\
&\gamma_3
=-\frac{(c^2-1)^2(c^2-2)(21c^4+16c^2+4)}{8c^4(c^2+1)^5}, \quad \gamma_4 =\frac{z^2}{2c}-\frac{3c^2+1}{6c(c^2+1)}.
\end{split}
\end{equation}
Therefore, it follows that
\begin{equation*}\label{u-equation-13a}
\begin{split}
\eta &=\eta_{00} + \varepsilon \eta_{10}+ \varepsilon^2 \eta_{20} + \mu \eta_{01}+ \varepsilon^3 \eta_{30} + \varepsilon\mu\eta_{11}+O(\varepsilon^4,\mu^2)\\
&=\frac{1}{c}u_{00}+\varepsilon\bigg(\frac{1}{c}u_{10}+\gamma_1 u^2_{00}\bigg)+\varepsilon^2\bigg(\frac{1}{c}u_{20}+ 2\gamma_1 u_{00}u_{10}+\gamma_2 u_{00}^3\bigg)\\
&+\mu \frac{1}{c}u_{01}+\varepsilon\mu\bigg(\frac{1}{c}u_{11} + 2 \gamma_1 u_{00}u_{01} +\gamma_4 u_{00,\xi\xi}\bigg)\\
& +\varepsilon^3\bigg( \frac{1}{c}u_{30} +\gamma_1 u_{10}^2+ 2 \gamma_1 u_{00}u_{20}+ 3 \gamma_2 u_{00}^2u_{10}+\gamma_3 u_{00}^4\bigg) +O(\varepsilon^4,\mu^2).
\end{split}
\end{equation*}
which along with $u =u_{00} + \varepsilon u_{10}+ \varepsilon^2 u_{20} + \mu u_{01}+ \varepsilon^3 u_{30} + \varepsilon\mu u_{11}+O(\varepsilon^4,\mu^2)$  yields
\begin{equation}\label{eta-u}
\begin{split}
  \eta = \frac{1}{c}u + \gamma_1 \varepsilon u^2 + \gamma_2\varepsilon^2u^3 + \gamma_3\varepsilon^3u^4 + \gamma_4\varepsilon\mu u_{\xi\xi} +O(\varepsilon^4,\mu^2),
\end{split}
\end{equation}
where $\gamma_i$ ($i=1, 2, 3, 4$) are defined in \eqref{gamma-defi-2} and the parameter $z \in [0, 1]$.

\begin{remark}\label{rmk-eta-form}
From the above derivation, we know that, in the free-surface incompressible irrotational Euler equations in the equatorial region, the relation between the free surface $\eta$ and the horizontal velocity $u$ formally obeys the equation \eqref{eta-u}, with or without Coriollis effect. It also illustrates that, all the classical models, such as the classical KdV equation, the BBM equation, or the (improved) Boussinesq equation, can be also formally derived from relation \eqref{eta-u} with the KdV regime $\varepsilon=O(\mu)$ in the equatorial region.
\end{remark}

In the following steps, we will derive the equation for $ u $ from express \eqref{etaepsilon3}. \\
In view of  \eqref{eta-u}, we have
\begin{equation}\label{etaepsilon3-2}
\begin{split}
2(\Omega + c) \eta_\tau =&\frac{2(\Omega + c)}{c} u_{\tau} + \frac{2(\Omega + c)(c_1+c)}{c^3}\varepsilon (u^2)_{\tau} + 2(\Omega + c)\gamma_2\varepsilon^2(u^3)_{\tau} \\
&+ 2(\Omega + c) \gamma_{3}\varepsilon^3(u^4)_{\tau}+ 2(\Omega + c)\gamma_4\varepsilon\mu u_{\tau\xi\xi} +O(\varepsilon^4,\mu^2),
\end{split}
\end{equation}
and
\begin{equation*}
\begin{split}
 3c^2\eta\eta_{\xi} &\;= \frac{3c^2}{2}\bigg((\frac{1}{c}u + \frac{c_1+c}{c^3}\varepsilon u^2 +  \gamma_2\varepsilon^2u^3 +
   \gamma_3 \varepsilon^3u^4)^2+ \gamma_4\varepsilon\mu u_{\xi\xi} \bigg)_{\xi}+O(\varepsilon^4,\mu^2) \\
 &\; = \frac{3c^2}{2}\bigg(\frac{1}{c^2}u^2 + \frac{2(c_1+c)}{c^4}\varepsilon u^3 +( \frac{(c_1+c)^2}{c^6} + \frac{2}{c}\gamma_2)\varepsilon^2u^4 +\frac{2}{c} \gamma_4 \mu\varepsilon uu_{\xi\xi} \\
 &\; \qquad\quad + (\frac{2}{c}\gamma_3 + \frac{2(c_1+c)}{c^3}\gamma_2)\varepsilon^3u^5\bigg)_{\xi}+O(\varepsilon^4,\mu^2). \\
\end{split}
\end{equation*}
Similarly, we may get
\begin{equation*}
  \frac{c^2}{3} \mu \eta_{\xi\xi\xi} =\frac{c^2}{3}\mu(\frac{1}{c}u + \frac{c_1+c}{c^3}\varepsilon u^2)_{\xi\xi\xi}+O(\varepsilon^4,\mu^2),
\end{equation*}
\begin{equation*}
\begin{split}
\varepsilon \mu\bigg(A_3 \eta_{\xi}\eta_{\xi\xi} +A_4 \eta\eta_{\xi\xi\xi}\bigg) =\varepsilon \mu\bigg(\frac{A_3 }{c^2}u_{\xi}u_{\xi\xi} +\frac{A_4 }{c^2} uu_{\xi\xi\xi}\bigg)+ O(\varepsilon^4,\mu^2),
\end{split}
\end{equation*}
\begin{equation*}
\begin{split}
 A_1\varepsilon\eta^2\eta_{\xi} =\frac{A_{1}}{3}\varepsilon\bigg[ \frac{1}{c^3}u^3 + \frac{3(c_1+c)}{c^5}\varepsilon u^4 + (\frac{3(c_1+c)^2}{c^7}+\frac{3}{c^2}\gamma_2)\varepsilon^2 u^5 \bigg]_{\xi}+O(\varepsilon^4,\mu^2),
\end{split}
\end{equation*}
\begin{equation*}
\begin{split}
 A_2\varepsilon^2\eta^3\eta_{\xi}  = \frac{A_2}{4c^4}\varepsilon^2(u^4)_{\xi} + \frac{A_2(c_1+c)}{c^6}\varepsilon^3(u^5)_{\xi}+O(\varepsilon^4,\mu^2),
\end{split}
\end{equation*}
and
\begin{equation*}
\begin{split}
-5B_2\varepsilon^3\eta^4\eta_{\xi} = -\frac{B_2}{c^5}\varepsilon^3(u^5)_{\xi}+O(\varepsilon^4,\mu^2).\\
\end{split}
\end{equation*}
Hence, we deduce from the equation \eqref{etaepsilon3} that
\begin{equation}\label{uepsilon3_A7}
\begin{split}
 u_{\tau}&\;+\frac{2(c_1+c)}{c^2}\varepsilon uu_{\tau} + 3\gamma_2 c\varepsilon^2u^2u_{\tau} + \gamma_4 c\varepsilon\mu u_{\tau\xi\xi}+4 \gamma_3 c\varepsilon^3u^3u_{\tau}
     + \frac{3c}{2(\Omega+c)}uu_{\xi}  \\
&\;+ \frac{cA_5}{2(\Omega+c)}\varepsilon^2u^3u_{\xi} + \frac{cA_6}{2(\Omega+c)}\varepsilon u^2u_{\xi} + \frac{c^2}{6(\Omega+c)}\mu u_{\xi\xi\xi} + \frac{cA_{7}}{2(\Omega+c)}\varepsilon^3 u^4u_{\xi}\\
&\;+ (\frac{cA_8}{2(\Omega+c)}u_{\xi}u_{\xi\xi} + \frac{cA_{9}}{2(\Omega+c)}uu_{\xi\xi\xi})\varepsilon\mu
= O(\varepsilon^4, \varepsilon^2\mu, \mu^2),
\end{split}
\end{equation}
where $A_5 :=  \frac{6(c_1+c)^2}{c^4} + 12c \gamma_2 + \frac{4A_1(c_1+c)}{c^5}\varepsilon^2 + \frac{A_2}{c^4}$,
$A_6 :=  \frac{9(c_1+c)}{c^2} + \frac{A_1}{c^3}$,
$A_8 :=  3c \gamma_4+ \frac{2(c_1+c)}{c} - \frac{A_3}{c^2}$,
$A_{9}:=   3c \gamma_4 + \frac{2(c_1+c)}{3c} - \frac{A_4}{c^2}$,
and $A_{7}:= 5\bigg[ \frac{3}{2}c^2(\frac{2 }{c}\gamma_3+\frac{2(c_1+c)}{c^3}\gamma_2) + \frac{A_1}{3}(\frac{3}{c^7}(c_1+c)^2+\frac{3}{c^2}\gamma_2) + \frac{A_2(c_1+c)}{c^6} - \frac{B_2}{c^5} \bigg]$.

Hence, we obtain
\begin{equation*}
\begin{split}
 \varepsilon uu_{\tau}&\; = -\varepsilon u \bigg( \frac{2(c_1+c)}{c^2}\varepsilon uu_{\tau} + 3 \gamma_2 c\varepsilon^2u^2u_{\tau} + \frac{3c}{2(\Omega+c)}uu_{\xi}
 + \frac{cA_5}{2(\Omega+c)}\varepsilon^2u^3u_{\xi}   \\
&\;\qquad\qquad+ \frac{cA_6}{2(\Omega+c)}\varepsilon u^2u_{\xi} + \frac{c^2}{6(\Omega+c)}\mu u_{\xi\xi\xi}\bigg)
+ O(\varepsilon^4, \varepsilon^2\mu, \mu^2),
\end{split}
\end{equation*}
which implies
\begin{equation*}
\begin{split}
 \varepsilon u\bigg(1+\frac{2(c_1+c)}{c^2}\varepsilon u &\;+3\gamma_2 c\varepsilon^2u^2\bigg)u_{\tau} = -\varepsilon u \bigg(  \frac{3c}{2(\Omega+c)}uu_{\xi}
 + \frac{cA_5}{2(\Omega+c)}\varepsilon^2u^3u_{\xi}   \\
&\;+ \frac{cA_6}{2(\Omega+c)}\varepsilon u^2u_{\xi} + \frac{c^2}{6(\Omega+c)}\mu u_{\xi\xi\xi}\bigg)
+ O(\varepsilon^4, \varepsilon^2\mu, \mu^2).
\end{split}
\end{equation*}
This follows that
\begin{equation*}
\begin{split}
\varepsilon uu_{\tau}&\; = -\varepsilon u \bigg[1- (\frac{2(c_1+c)}{c^2}\varepsilon u+3\gamma_2 c\varepsilon^2u^2) + (\frac{2(c_1+c)}{c^2}\varepsilon u)^2 \bigg]
    \bigg[ \frac{3c}{2(\Omega+c)}uu_{\xi}  \\
&\;\quad+ \frac{cA_5}{2(\Omega+c)}\varepsilon^2u^3u_{\xi} + \frac{cA_6}{2(\Omega+c)}\varepsilon u^2u_{\xi} + \frac{c^2}{6(\Omega+c)}\mu u_{\xi\xi\xi}\bigg]
    + O(\varepsilon^4, \mu^2),
\end{split}
\end{equation*}
and then
\begin{equation}\label{uuT}
\begin{split}
\varepsilon uu_{\tau}&\; = -\varepsilon u\bigg[ \frac{3c}{2(\Omega+c)}uu_{\xi} + \frac{c^2}{6(\Omega+c)}\mu u_{\xi\xi\xi}
  + \frac{c^2A_6-6(c_1+c)}{2c(\Omega+c)}\varepsilon u^2u_{\xi} \\
&\;\quad + \frac{c^2A_5 - 2A_6(c_1+c) + 3c^2(\frac{4(c_1+c)^2}{c^4}-3 \gamma_2 c)}{2c(\Omega+c)}\varepsilon^2u^3u_{\xi} \bigg]+ O(\varepsilon^4, \mu^2),
\end{split}
\end{equation}
\begin{equation} \label{u3uT}
\begin{split}
\varepsilon^2 u^2u_{\tau} =&\; -\varepsilon^2 u^2\bigg[ \frac{3c}{2(\Omega+c)}uu_{\xi}  + \frac{c^2A_6-6(c_1+c)}{2c(\Omega+c)}\varepsilon u^2u_{\xi} \bigg]
  + O(\varepsilon^4, \varepsilon^2\mu, \mu^2), \\
\varepsilon^3 u^3u_{\tau} =&\; - \frac{3c}{2(\Omega+c)}\varepsilon^3 u^4 u_{\xi}  + O(\varepsilon^4, \mu^2), \quad
\varepsilon\mu  u_{\tau\xi\xi}=- \frac{3c}{2(\Omega+c)} \varepsilon\mu   (u u_{\xi})_{\xi\xi} + O(\varepsilon^4, \mu^2)
\end{split}
\end{equation}
Decompose $\varepsilon\mu u_{\tau\xi\xi}$ into  $\varepsilon\mu (1-\nu)u_{\tau\xi\xi} + \varepsilon\mu\nu u_{\tau\xi\xi}$ for some constant $\nu$ (to be determined later), we may get from \eqref{u3uT} that
\begin{equation} \label{nu}
\begin{split}
\varepsilon\mu u_{\tau\xi\xi}= \varepsilon\mu (1-\nu)u_{\tau\xi\xi} - \frac{3c \nu}{2(\Omega+c)} \varepsilon\mu   (u u_{\xi})_{\xi\xi} + O(\varepsilon^4, \mu^2).
\end{split}
\end{equation}
Substituting \eqref{uuT}-\eqref{nu} into \eqref{uepsilon3_A7}, we obtain that
\begin{equation*}
\begin{split}
 u_{\tau}&\;+ c\gamma_4(1-\nu)\mu\varepsilon u_{\tau\xi\xi} + \frac{3c}{2(\Omega+c)}uu_{\xi}+\frac{c^2}{6(\Omega+c)}\mu u_{\xi\xi\xi}-\frac{9c^2 \gamma_2}{2(\Omega+c)}\varepsilon^2u^3u_{\xi}\\
&\; - \frac{3c^2\gamma_4\nu}{2(\Omega+c)}\mu\varepsilon(u u_{\xi})_{\xi\xi}  + \frac{2(c_1+c)}{c^2}\varepsilon\bigg[\frac{3c}{2(\Omega+c)}u^2u_{\xi}
+ \frac{c^2}{6(\Omega+c)}\mu uu_{\xi\xi\xi} \\
&\; + \frac{c^2A_6-6(c_1+c)}{2c(\Omega+c)}\varepsilon u^3u_{\xi}\bigg]+ \frac{cA_5}{2(\Omega+c)}\varepsilon^2u^3u_{\xi} + \frac{cA_6}{2(\Omega+c)}\varepsilon u^2u_{\xi} \\
&\; + \mu\varepsilon(\frac{cA_8}{2(\Omega+c)}u_{\xi}u_{\xi\xi} + \frac{cA_{9}}{2(\Omega+c)}uu_{\xi\xi\xi}) + A_{10}\varepsilon^3 u^4u_{\xi}
=  O(\varepsilon^4, \mu^2),
\end{split}
\end{equation*}
where
\begin{equation*}
\begin{split}
A_{10}:=&\; \frac{cA_{7}}{2(\Omega+c)}-\frac{(c_1+c)\bigg(c^2A_5 - 2A_6(c_1+c) + 3c^2(\frac{4(c_1+c)^2}{c^4}-3\gamma_2c)\bigg)}{c^3(\Omega+c)} \\
&\;- \frac{3 \gamma_2(c^2A_6-6(c_1+c))+12c^2\gamma_3}{2(\Omega+c)},
\end{split}
\end{equation*}
which implies
\begin{equation}\label{uepsi3abb}
\begin{split}
u_{\tau}&\; + \frac{3c^2}{c^2+1}uu_{\xi}+\frac{c^3}{3(c^2+1)}\mu u_{\xi\xi\xi} + c \gamma_4(1-\nu)\mu\varepsilon u_{\tau\xi\xi}
 + A_{11}\varepsilon u^2u_{\xi} \\
&\; + A_{12}\varepsilon^2u^3u_{\xi}
 + A_{10}\varepsilon^3u^4u_{\xi}  + \mu\varepsilon\bigg[A_{13}uu_{\xi\xi\xi} + A_{14}u_{\xi}u_{\xi\xi}\bigg] = O(\varepsilon^4, \varepsilon^2\mu, \mu^2). \\
\end{split}
\end{equation}
where
$A_{11} := \frac{c^2A_6-6(c_1+c)}{2c(\Omega+c)}=\frac{-3c(c^2-1)(c^2-2)}{2(c^2+1)^3}$,
$A_{12} := \frac{cA_5}{2(\Omega+c)} - \frac{9c^2 \gamma_2}{2(\Omega+c)} - \frac{2(c_1+c)}{c^2}\frac{c^2A_6-6(c_1+c)}{2c(\Omega+c)}=\frac{(c^2-1)^2(c^2-2)(8c^2-1)}{2(c^2+1)^5}$,
$A_{13} := \frac{cA_{9}}{2(\Omega+c)}- \frac{3c^2 \gamma_4 \nu}{2(\Omega+c)} - \frac{c_1+c}{3(\Omega+c)}=\frac{3c^3 \gamma_4}{(c^2+1)}(1-\nu) +\frac{c^2(3c^4+8c^2-1)}{3(c^2+1)^3}$,
$A_{14} := \frac{cA_8}{2(\Omega+c)} - \frac{9c^2\gamma_4\nu}{2(\Omega+c)}= \frac{3c^3}{(c^2+1)}\gamma_4(1-3\nu)+\frac{c^2(6c^4+19c^2+4)}{3(c^2+1)^3}$.

Back to the original transformation
$ x=\varepsilon^{-\frac{1}{2}}\xi+c\varepsilon^{-\frac{3}{2}}\tau,\quad  t = \varepsilon^{-\frac{3}{2}}\tau$,
we have
\begin{equation*}
\begin{split}
\frac{\partial}{\partial\xi}  = \varepsilon^{-\frac{1}{2}}\partial_x, \quad \frac{\partial}{\partial \tau} = \varepsilon^{-\frac{3}{2}}(c\partial_x+ \partial_t).
\end{split}
\end{equation*}
Hence, according to this transformation, the equation \eqref{uepsi3abb} can be written as
\begin{equation*}
\begin{split}
&\; u_t  + c u_x +  \frac{3c^2}{c^2+1}\varepsilon uu_{x} + A_{11}\varepsilon^2 u^2u_{x}+ A_{12}\varepsilon^3u^3u_{x}+ c\gamma_4(1-\nu)\mu u_{txx} \\
&\; + \Big(\frac{c^3}{3(c^2+1)} - c^2 \gamma_4(1-\nu)\Big)\mu u_{xxx}
+ \mu\varepsilon\bigg(A_{13}uu_{xxx} + A_{14}u_{x}u_{xx}\bigg)  = O(\varepsilon^4, \mu^2).
\end{split}
\end{equation*}

In order to get the R-CH equation, we
need
\begin{equation*}
\begin{split}
&\frac{2c^2}{(c^2+1)}c \gamma_4(1-\nu)=2A_{13}=A_{14},
\end{split}
\end{equation*}
which yields
\begin{equation}\label{gamma-4-1}
\begin{split}
 &\frac{2c^3}{(c^2+1)}\gamma_4 =\frac{-c^2(3c^4+6c^2-5)}{6(c^2+1)^3}
\end{split}
\end{equation}
and then
\begin{equation*}
\begin{split}
\frac{2c^2}{(c^2+1)}c \gamma_4(1-\nu)=2A_{13}=A_{14}=\frac{-c^2(3c^4+8c^2-1)}{3(c^2+1)^3}.
\end{split}
\end{equation*}
Therefore, it enables us to derive the R-CH equation in the form
\begin{equation*}\label{R-CH-1-1}
\begin{split}
 u_t -\beta\mu  u_{xxt} + c u_x + 3\alpha\varepsilon uu_x - \beta_0\mu u_{xxx} &+ \omega_1 \varepsilon^2u^2u_x + \omega_2 \varepsilon^3u^3u_x   \\
&= \alpha\beta\varepsilon\mu( 2u_{x}u_{xx}+uu_{xxx}).
\end{split}
\end{equation*}
Combining \eqref{gamma-4-1} and \eqref{gamma-defi-2}, it is found  that the height parameter $z$ in $\gamma_4$ may take the value
\begin{equation}\label{height-2}
z_0 =\bigg(\frac {1}{2} - \frac{2}{3} \frac{1}{(c^2 + 1)} + \frac{4}{3} \frac{1}{(c^2 + 1)^2}\bigg)^{1/2}.
\end{equation}

\renewcommand{\theequation}{\thesection.\arabic{equation}}
\setcounter{equation}{0}


\section{Local well-posedness} \label{local}

Our attention in this section is now turned to the local-posedness issue for the  R-CH equation.
Recall the R-CH equation \eqref{R-CH-1} in terms of the evolution of $m$, namely, the equation \eqref{R-CH-m}. Applying the transformation
$u_{\varepsilon, \mu}(t, x) = \alpha \varepsilon u(\sqrt{\beta \mu}\,t,\sqrt{\beta \mu}\,x)$  to \eqref{R-CH-m},
we know that $u_{\varepsilon, \mu}(t, x)$ solves
\begin{equation*}
\begin{split}
u_t -  u_{xxt} + c u_x + 3 u u_x - \frac{\beta_0}{\beta} u_{xxx} +\frac{\omega_1}{\alpha^2}u^2 u_x+ \frac{\omega_2}{\alpha^3}u^3 u_x = 2 u_{x} u_{xx} + u u_{xxx},
\end{split}
\end{equation*}
and its corresponding three conserved quantities (still denoted by $I(u)$, $E(u)$, and $F(u)$) are as follows
\[
I(u) =\Int u\, dx, \quad  E(u)=\frac{1}{2}\Int u^2+u_x^2\,dx,
\] and
\[
 F(u)=\frac{1}{2}\Int c u^2+ u^3+\frac{\beta_0}{\beta}u_x^2+\frac{\omega_1}{6\alpha^2}u^4 + \frac{\omega_2}{10\alpha^3}u^5 +uu^2_x\,dx.
\]
And we also have two more forms of equations,
\begin{equation*}
\begin{cases}
m_t + u m_x + 2 u_x m +  c u_x - \frac{\beta_0}{\beta} u_{xxx} +\frac{\omega_1}{\alpha^2}u^2 u_x+ \frac{\omega_2}{\alpha^3}u^3 u_x = 0,\\
m = u - u_{xx},
\end{cases}
\end{equation*}
and
\begin{equation}\label{weak-RCH}
u_t + u u_x+\frac{\beta_0}{\beta}u_x + p * \partial_x\left\{\left(c-\frac{\beta_0}{\beta}\right)u + u^2+\frac{1}{2}u_x^2+\frac{\omega_1}{3\alpha^2}u^3+\frac{\omega_2}{4\alpha^3}u^4\right\}=0.
\end{equation}
where $p = \frac{1}{2}e^{-|x|}$.

Now we are in a position to state the local well-posedness result of the following Cauchy problem, which may be similarly obtained as in \cite{CL09, Danchin01}
(up to a slight modification).
\begin{equation}\label{rCH-Cauchy}
\begin{cases}
u_t -  u_{xxt} + c u_x + 3 u u_x - \frac{\beta_0}{\beta} u_{xxx} +\frac{\omega_1}{\alpha^2}u^2 u_x+ \frac{\omega_2}{\alpha^3}u^3 u_x = 2 u_{x} u_{xx} + u u_{xxx},\\
u|_{t = 0} = u_0.
\end{cases}
\end{equation}


\begin{theorem}\label{local}
Let $u_0 \in H^{s}(\mathbb{R})$ with $s > \frac{3}{2}$.  Then there exist a positive time $T>0$ and a unique solution $u \in C([0, T]; H^s(\mathbb{R})) \cap  C^1([0, T]; H^{s-1}(\mathbb{R}))$ to the Cauchy problem \eqref{rCH-Cauchy} with $u(0)=u_0$. Moreover, the solution $u$ depends continuously on the
initial value $u_0$. In addition, the Hamiltonians $I(u)$, $E(u)$ and $F(u)$ are independent of the existence time $t>0$.
\end{theorem}

Thanks to the scaling of the solution $u_{\varepsilon, \mu}(t, x) = \alpha \varepsilon u(\sqrt{\beta \mu}\,t,\sqrt{\beta \mu}\,x)$,  the large existence time for
Equation \eqref{R-CH-1} has the form $ \frac{T}{\epsilon}.$

Motivated to the method in \cite{Danchin01}, the following blow-up criterion can be also derived, and we omit details of its proof.
\begin{theorem}[Blow-up criterion]
Let $s > \frac{3}{2}$, $u_0 \in H^s$ and $u$ be the corresponding solution to \eqref{rCH-Cauchy} as in Theorem \ref{local}. Assume $T^*_{u_0}$ is the maximal time of existence. Then
\begin{equation}\label{blowup-criterion-1}
T^{\ast}_{u_0} < \infty \quad \Rightarrow \quad \int_0^{T^{\ast}_{u_0}} \|\partial_x u(\tau)\|_{L^{\infty}} d\tau = \infty.
\end{equation}
\end{theorem}
\begin{remark}\label{rmk-blowup-1}
The blow-up criterion \eqref{blowup-criterion-1} implies that
the lifespan $T^{\ast}_{u_0}$ does not depend on the regularity index $s$ of the initial data $u_0$.
\end{remark}

Now we return to the original R-CH \eqref{R-CH-1}, and let
\begin{equation*}
\|u\|^2_{X^{s+1}_{\mu}} = \|u\|^2_{H^s} + \mu \beta \|\partial_x u\|^2_{H^s}.
\end{equation*}
For some $\mu_0 > 0$ and $M > 0$, we define the Camassa-Holm regime $\mathcal{P}_{\mu_0, M} := \{(\varepsilon, \mu): <\mu \leq \mu_0, 0<\varepsilon \leq M \sqrt{\mu}\}$. Then, we have the following corollary.
\begin{cor}(\cite{CL09})
Let $u_0 \in H^{s+1}(\mathbb{R})$, $\mu_0 > 0$ and $M > 0$, $s > \frac{3}{2}$. Then, there exist $T > 0$ and a unique family of solutions $\left(u_{\varepsilon,\mu}\right)|_{(\varepsilon,\mu) \in \mathcal{P}_{\mu_0, M}}$  in $C\left(\left[0,\frac{T}{\varepsilon}\right];X^{s+1}(\mathbb{R})\right) \cap C^1\left(\left[0,\frac{T}{\varepsilon};X^s(\mathbb{R})\right]\right)$ to the Cauchy problem
\begin{equation*}
\begin{cases}
&\partial_t u-\beta\mu \partial_t u_{xx} + c u_x + 3\alpha\varepsilon uu_x - \beta_0\mu u_{xxx} + \omega_1 \varepsilon^2u^2u_x + \omega_2 \varepsilon^3u^3u_x   \\
&\qquad\qquad \qquad\qquad \qquad\qquad \qquad\qquad \qquad= \alpha\beta\varepsilon\mu( 2u_{x}u_{xx}+uu_{xxx}),\\
&u|_{t = 0} = u_0.
\end{cases}
\end{equation*}
\end{cor}

\renewcommand{\theequation}{\thesection.\arabic{equation}}
\setcounter{equation}{0}

\section{Wake-breaking phenomena} \label{breaking}

Using the energy estimates,  we can further obtain the following wave breaking criterion  to the R-CH equation.

\begin{theorem}[Wave breaking criterion]\label{thm-wavebreak-crt}
Let $u_0 \in H^s(\mathbb{R})$ with $s > \frac{3}{2}$, and $T^{\ast}_{u_0} >0$ be the
maximal existence time of the solution $u$ to the system
\eqref{rCH-Cauchy}  with initial data $u_0$ as in Theorem \ref{local}. Then the corresponding solution blows up in finite time if and only if
\begin{equation}\label{wavw-breaking-condition}
\liminf_{t \uparrow T^{\ast}_{u_0}, x \in \mathbb{R}} u_x(t, x) = - \infty.
\end{equation}
\end{theorem}
\begin{proof}
Applying Theorem \ref{local}, Remark \ref{rmk-blowup-1}, and a simple
density argument, we only need to show that Theorem \ref{thm-wavebreak-crt}
holds for some $s \geq 3.$ Here we assume $s = 3$ to prove the
above theorem.

Multiplying the first equation in \eqref{rCH-Cauchy} by $u$ and
integrating by parts, we get
\begin{equation}\label{4.1-1}
\frac{1}{2}\frac{d}{dt}\|u\|_{H^1}^2=0,
\end{equation}
and then for any $t\in (0, T^{\ast}_{u_0})$
\begin{equation}\label{4.1-1c}
\|u(t)\|_{H^1}=\|u_0\|_{H^1}.
\end{equation}
On the other hand, multiplying the first equation in
\eqref{rCH-Cauchy}  by $u_{xx}$ and integrating by parts again, we
obtain
\begin{equation}\label{4.4-1b}
\begin{split}
\frac{1}{2}\frac{d}{dt}\|u_x\|_{H^1}^2&=-\frac{3}{2}\int_{\mathbb{R}}u_{x}(u_{x}^2+u_{xx}^2)
dx-\int_{\mathbb{R}} (\frac{\omega_1}{\alpha^2}u^2u_x+\frac{\omega_2}{\alpha^3}u^3u_x)u_{xx} \,dx\\
&=-\frac{3}{2}\int_{\mathbb{R}}u_{x}(u_{x}^2+u_{xx}^2)
dx+\int_{\mathbb{R}} \big|\frac{\omega_1}{2\alpha^2}u^2+\frac{\omega_2}{2\alpha^3}u^3\big|(u_x^2+u_{xx}^2) \,dx.
\end{split}
\end{equation}
Assume that $T^{\ast}_{u_0} < +\infty$ and  there exists $M > 0$ such that
\begin{equation}\label{4.4-1-1a}
u_x(t, x) \geq -M, \quad \forall \, (t, \, x) \in [0, T^{\ast}_{u_0}) \times
\mathbb{R}.
\end{equation}
 It then follows from \eqref{4.1-1}, \eqref{4.1-1c}, and \eqref{4.4-1b} that
\begin{equation}\label{4.9-1}
\begin{split}
\frac{d}{dt}\int_{\mathbb{R}}(u^2+2u_{x}^2+u_{xx}^2)
dx &\leq (\frac{3}{2}M+\frac{|\omega_1|}{2\alpha^2}\|u\|_{L^{\infty}}^2+\frac{|\omega_2|}{2|\alpha|^3}\|u\|_{L^{\infty}}^3) \int_{\mathbb{R}}(u_{x}^2+u_{xx}^2)\,dx\\
& \leq C(1+M+\|u\|_{H^1}^3)\int_{\mathbb{R}}(u_{x}^2+u_{xx}^2)\,dx,
\end{split}
\end{equation}
where we used the Sobolev embedding theorem $H^{s}(\mathbb{R})
\hookrightarrow L^{\infty}(\mathbb{R})$ (with $s>\frac{1}{2}$) in the last inequaity.
Applying Gronwall's inequality to \eqref{4.9-1} yields for every
$t \in [0, \, T^{\ast}_{u_0})$
\begin{equation}\label{4.10-1}
\|u(t)\|_{H^{2}}^2 \leq
2\|u_0\|_{H^{2}(\mathbb{R})}^2 e^{C t(1+M+\|u_0\|_{H^{1}}^3)}
\leq
2\|u_0\|_{H^{2}(\mathbb{R})}^2 e^{C T^{\ast}_{u_0}(1+M+\|u_0\|_{H^{1}}^3)} .
\end{equation}
Differentiating the first equation in \eqref{rCH-Cauchy} with
respect to $x$, and  multiplying the result equation by $u_{xxx},$
then integrating by parts, we get
\begin{equation*}\label{4.11-1}
\begin{split}
& \frac{1}{2}\frac{d}{dt}\int_{\mathbb{R}}(u_{xx}^2+u_{xxx}^2)\,
dx\\
&=-\frac{15}{2}\int_{\mathbb{R}}u_{x}u_{xx}^2
dx-\frac{5}{2}\int_{\mathbb{R}}u_{x}u_{xxx}^2
dx-\int_{\mathbb{R}} (\frac{\omega_1}{\alpha^2}u^2u_x+\frac{\omega_2}{\alpha^3}u^3u_x)_{x}u_{xxx} \,dx\\
& \leq C(1+M+\|u\|_{L^{\infty}}^3) \int_{\mathbb{R}}(u_{xx}^2+u_{xxx}^2) \,dx+C(\|u\|_{L^{\infty}}^2+\|u\|_{L^{\infty}}^4)\|u_x\|_{L^4}^4,
\end{split}
\end{equation*}
where we have used the assumption \eqref{4.4-1-1a}, which follows from the Sobolev embedding theorem and the interpolation inequality
$\|f\|_{L^4(\mathbb{R})} \leq C \|f\|_{L^2(\mathbb{R})}^{\frac{3}{4}} \|f_x\|_{L^{2}(\mathbb{R})}^{\frac{1}{4}} $
 that
\begin{equation*}\label{4.11-1a}
\begin{split}
&\frac{d}{dt}\int_{\mathbb{R}}(u_{xx}^2+u_{xxx}^2)\,
dx\leq C(1+M+\|u_0\|_{H^{1}}^3) \int_{\mathbb{R}}(u_{xx}^2+u_{xxx}^2) \,dx\\
&\qquad\qquad\qquad \qquad\qquad+C\|u_0\|_{H^{1}}^5(1+\|u_0\|_{H^{1}}^2)\|u_{xx}\|_{L^2}\\
&\leq  C(1+M+\|u_0\|_{H^{1}}^{14}) \int_{\mathbb{R}}(u_{xx}^2+u_{xxx}^2) \,dx.
\end{split}
\end{equation*}
Hence, Gronwall's
inequality applied implies that for every
$t \in [0, \, T^{\ast}_{u_0})$
\begin{equation*}\label{4.14-1}
\begin{split}
 \int_{\mathbb{R}}(u_{xx}^2+u_{xxx}^2)
dx \leq e^{C(1+M+\|u_0\|_{H^{1}}^{14})  T^{\ast}_{u_0}}\int_{\mathbb{R}}(u_{0xx}^2+u_{0xxx}^2)
dx,
\end{split}
\end{equation*}
which, together with \eqref{4.10-1}, yields that for every $t
\in [0, \, T^{\ast}_{u_0})$,
\begin{equation*}
\|u(t)\|_{H^{3}(\mathbb{R})}^2
\leq
3\|u_0\|_{H^{3}(\mathbb{R})}^2 e^{C (1+M+\|u_0\|_{H^{1}}^{14})  T^{\ast}_{u_0}}.
\end{equation*}
This contradicts the assumption the maximal existence time
$T^{\ast}_{u_0}<+\infty.$

Conversely, the Sobolev embedding theorem $H^{s}(\mathbb{R})
\hookrightarrow L^{\infty}(\mathbb{R})$ (with $s>\frac{1}{2}$)
implies that if \eqref{wavw-breaking-condition} holds, the corresponding solution blows
up in finite time, which completes the proof of Theorem
\ref{thm-wavebreak-crt}.
\end{proof}

Recall the R-CH equation \eqref{weak-RCH}, namely,
\begin{equation*}
u_t + u u_x+\frac{\beta_0}{\beta}u_x + p_x \ast \left (\left(c-\frac{\beta_0}{\beta}\right)u + u^2+\frac{1}{2}u_x^2+\frac{\omega_1}{3\alpha^2}u^3 + \frac{\omega_2}{4\alpha^3}u^4 \right )=0,
\end{equation*}
where $p = \frac{1}{2}e^{-|x|}$. The wave breaking phenomena  could be now illustrated by choosing certain the initial data.

\begin{theorem}[Wave breaking data]\label{Blow-up}
Suppose $u_0 \in H^s$ with $s > 3/2$. Let $T > 0$ be the maximal time of existence of the corresponding solution $u(t, x)$ to \eqref{weak-RCH} with the initial data $u_0$. Assume these is  $x_0 \in \mathbb{R}$ such that
\begin{equation*}
u_{0,x}(x_0) < - \left | u_0(x_0) - \frac{1}{2} \left ( \frac{\beta_0}{\beta} - c \right ) \right  |  -\sqrt{2}C_0,
\end{equation*}
where $ C_0 > 0 $ is defined by
\begin{equation} \label{Blow-up data}
C_0^2 = \frac{|\omega_1|}{2 \alpha^2} E_0^{\frac{3}{2}} + \frac{ |\omega_2|}{2 \alpha^3} E_0^2,
\end{equation}
and
$$
E_0 = \frac{1}{2} \int_{\mathbb{R}} \left ( u_0^2 + (\partial_xu_0)^2 \right )  dx.
$$
Then the solution $u(t, x)$ breaks down  at the time
\begin{equation*}
T \leq \frac{2}{\sqrt{u_{0,x}^2(x_0) -  \left (u_0(x_0) - \frac{1}{2}  \left ( \frac{\beta_0}{\beta} - c \right ) \right  )^2 }-\sqrt{2}C_0}.
\end{equation*}
\end{theorem}
\begin{remark}
In the case of  the rotation frequency $ \Omega = 0, $ or the wave speed $ c = 1,  $ the corresponding constant $ C_0 $ in \eqref {Blow-up data} must be zero, because the  parameters $ \omega_1 $ and $ \omega _2 $ vanish. The assumption on the wave breaking is then back to the case of the classical CH equation.
\end{remark}

\begin{proof}  Applying the translation $ u(t, x)  \mapsto u(t, x - \frac{\beta_0}{\beta} t) $ to equation \eqref{weak-RCH} yields the equation in the form,
\begin{equation} \label{weak-RCH-2}
u_t + u u_x + p_x \ast \left  (\left(c-\frac{\beta_0}{\beta}\right)u + u^2+\frac{1}{2}u_x^2+\frac{\omega_1}{3\alpha^2}u^3 + \frac{\omega_2}{4\alpha^3}u^4 \right  )=0.
\end{equation}
Taking the derivative  $\partial_x$ to \eqref{weak-RCH-2}, we have
\begin{equation}\label{u-xt}
\begin{split}
u_{xt} + u u_{xx} = & - \frac{1}{2}u^2_x + u^2 + \left(c-\frac{\beta_0}{\beta}\right)u+\frac{\omega_1}{3\alpha^2}u^3 + \frac{\omega_2}{4\alpha^3}u^4  \\
&- p  \ast \left (\left(c-\frac{\beta_0}{\beta}\right)u + u^2+\frac{1}{2}u_x^2+\frac{\omega_1}{3\alpha^2}u^3 + \frac{\omega_2}{4\alpha^3}u^4 \right ).
\end{split}
\end{equation}
We introduce the associated Lagrangian scales of \eqref{weak-RCH-2} as
\begin{equation*}
\begin{cases}
\frac{\partial q}{\partial t} = u(t, q), & 0 < t < T,\\
q(0, x) = x, & x \in \mathbb{R},
\end{cases}
\end{equation*}
where $u \in C^1([0,T), H^{s-1})$ is the solution to equation \eqref{weak-RCH-2} with initial data $u_0 \in H^s$, $s> 3/2$. Along with the trajectory of $q(t, x_))$, \eqref{weak-RCH-2} and \eqref{u-xt} become
\begin{gather*}
\frac{\partial u(t,q)}{\partial t} = - p_x \ast \left ( \left(c-\frac{\beta_0}{\beta}\right)u + u^2+\frac{1}{2}u_x^2+\frac{\omega_1}{3\alpha^2}u^3  +\frac{\omega_2}{4\alpha^3}u^4 \right ),\\
\begin{split}
\frac{\partial u_x(t,q)}{\partial t} = - \frac{1}{2}u^2_x + u^2 + & \left(c-\frac{\beta_0}{\beta}\right)u + \frac{\omega_1}{3\alpha^2}u^3  +\frac{\omega_2}{4\alpha^3}u^4 \\
& - p \ast \left (  \left(c-\frac{\beta_0}{\beta}\right)u + u^2+\frac{1}{2}u_x^2+\frac{\omega_1}{3\alpha^2}u^3  +\frac{\omega_2}{4\alpha^3}u^4  \right ).
\end{split}
\end{gather*}
Denote now at $ (t, q(t, x_0)),$
\begin{equation*}
M(t) = u(t,q) - \frac{k}{2} - u_x(t, q) \quad \text{and} \quad N(t) = u(t, q) - \frac{k}{2} +  u_x(t, q),
\end{equation*}
where $ k = \frac{\beta_0}{\beta} - c. $ Recall the two convolution operators $p_+$, $p_-$ as
\begin{equation*}
\begin{split}
& p_+ \ast f (x) = {e^{-x} \over 2} \int^x_{-\infty} e^y f(y) dy,\\
& p_- \ast f(x) = {e^{x}\over 2} \int^\infty_{x} e^{-y} f(y) dy
\end{split}
\end{equation*}
and the relation
\begin{equation*}
p = p_+ + p_-, \qquad p_x = p_- - p_+.
\end{equation*}
Applying \cite[Lemma 3.1 (1)]{BrCo2} with $m = - k^2/4$ and $K = 1$ we have the following convolution estimates
\begin{equation*}
p_\pm \ast \left ( u^2 - ku + {1\over 2}u^2_x \right )  \geq {1\over 4} \left ( u^2 - ku - {k^2 \over 4} \right ).
\end{equation*}
It then follows that  at $(t, q(t, x_0))$,
\begin{equation*}
\begin{split}
\frac{\partial M}{\partial t} =&\, \frac{1}{2} u^2_x - u^2 + k u - \frac{\omega_1}{3\alpha^2}u^3  - \frac{\omega_2}{4\alpha^3}u^4 \\
&\, + 2 p_{+} \ast \left (- ku + u^2+\frac{1}{2}u_x^2+\frac{\omega_1}{3\alpha^2}u^3 + \frac{\omega_2}{4\alpha^3}u^4 \right )\\
\geq  & \,  \frac{1}{2} \left ( u_x^2 - \left (u - \frac{k}{2} \right )^2 \right )   - \frac{\omega_1}{3\alpha^2}u^3 - \frac{\omega_2}{4\alpha^3}u^4 + 2 p_{+} \ast \left(\frac{\omega_1}{3\alpha^2}u^3 + \frac{\omega_2}{4\alpha^3}u^4 \right ) \\
= &\, -\frac{1}{2} MN    - \frac{\omega_1}{3\alpha^2}u^3  - \frac{\omega_2}{4\alpha^3}u^4  + 2 p_{+} \ast \left ( \frac{\omega_1}{3\alpha^2}u^3 + \frac{\omega_2}{4\alpha^3}u^4 \right )
\end{split}
\end{equation*}
\begin{equation*}
\begin{split}
\frac{\partial N}{\partial t} = &\, - \frac{1}{2} u^2_x + u^2 - k u + \frac{\omega_1}{3\alpha^2}u^3 + \frac{\omega_2}{4\alpha^3}u^4 \\
&\, - 2 p_{-} \ast \left ( - k u + u^2+\frac{1}{2}u_x^2+\frac{\omega_1}{3\alpha^2}u^3 +\frac{\omega_2}{4\alpha^3}u^4 \right )\\
\leq &\, -\frac{1}{2}\left ( u_x^2 - \left ( u - \frac{k}{2} \right )^2 \right )  + \frac{\omega_1}{3\alpha^2}u^3 +\frac{\omega_2}{4\alpha^3}u^4
- 2 p_{-} * \left ( \frac{\omega_1}{3\alpha^2}u^3 + \frac{\omega_2}{4\alpha^3}u^4 \right ) \\
=&\, \frac{1}{2}MN    + \frac{\omega_1}{3\alpha^2}u^3 + \frac{\omega_2}{4\alpha^3}u^4  - 2 p_{-} \ast \left (\frac{\omega_1}{3\alpha^2}u^3 + \frac{\omega_2}{4\alpha^3}u^4 \right )
\end{split}
\end{equation*}
The terms with  $ \omega_1 $ and $ \omega_2 $ in the right sides of the above estimates can be bounded by
\begin{equation*}
\begin{split}
\left| \frac{\omega_1}{3\alpha^2}u^3 \right . & \left .  +  \frac{\omega_2}{4\alpha^3}u^4  \mp 2p_{\pm} \ast \left (\frac{\omega_1}{3\alpha^2}u^3 + \frac{\omega_2}{4\alpha^3}u^4 \right ) \right|\\
\leq & \, \frac{|\omega_1|}{3\alpha^2} \|u\|_{L^\infty}^3  + \frac{|\omega_2|}{4\alpha^3} \|u\|_{L^\infty}^4 +  \|u\|_{L^\infty} \left(\frac{|\omega_1|}{3\alpha^2} \|u\|^2_{L^2}\right) +  \|u\|^2_{L^\infty} \left(\frac{|\omega_2|}{4\alpha^3} \|u\|^2_{L^2}\right) \\
\leq&\, \frac {|\omega_1| }{ 2 \alpha^2} E_0^{\frac{3}{2}} + \frac{|\omega_2|} {2 \alpha^3} E_0^2  = C^2_0 > 0,
\end{split}
\end{equation*}
where use has been made of the fact that
\begin{equation*}
\| p_\pm \|_{L^\infty} = {1\over 2}, \quad \| p_\pm \|_{L^2} = {1\over 2\sqrt{2}}.
\end{equation*}
In consequence, we have
\begin{equation}\label{AB}
\begin{cases}
\frac{d M}{d t} \geq - \frac{1}{2}MN - C^2_0,\\
\frac{d N}{d t} \leq \frac{1}{2}MN + C^2_0.
\end{cases}
\end{equation}
By the assumptions on $ u_0(x_0) $, it is easy to see that
\begin{equation*}
M(0) = u_0(x_0) - \frac{k}{2}  - u_{0,x}(x_0) > 0, \;
N(0) = u_0(x_0) - \frac{k}{2}  + u_{0,x}(x_0) < 0, \;   \frac{1}{2}M(0)N(0) + C^2_0 < 0. \label{AB-0}
\end{equation*}
By the continuity of $M(t)$ and $N(t)$, it then ensures that
\begin{equation*}
\frac{d M}{d t} > 0,\quad \frac{d N}{d t} < 0,\quad \forall t \in [0, T). \label{AB-1}
\end{equation*}
This in turn implies that
\begin{equation*}
M(t) > M(0) > 0,\quad N(t) < N(0) < 0, \quad \forall t \in [0,T). \label{AB-2}
\end{equation*}
Let $h(t) = \sqrt{-M(t)N(t)}$. It then follows from \eqref{AB} that
\begin{equation*}
\begin{split}
\frac{d h}{d t} =\frac{-M'(t)N(t) -M(t) N'(t)}{2 h}\geq &\,\frac{\left(-\frac{1}{2}MN-C^2_0\right)(-N)-M\left(\frac{1}{2}MN+C^2_0\right)}{2h}\\
=&\,\frac{M-N}{2h}\left(-\frac{1}{2}MN-C^2_0\right).
\end{split}
\end{equation*}
Using the estimate $\frac{M-N}{2h} \geq 1$ and  the fact that $h+\sqrt{2}C_0 > h - \sqrt{2}C_0 > 0$, we  obtain the following differential inequalities
\begin{equation*}
\begin{split}
\frac{d h}{d t} \geq&-\frac{1}{2}MN -C^2_0
=\frac{1}{2}(h-\sqrt{2}C_0)(h+\sqrt{2}C_0)
\geq \frac{1}{2}(h - \sqrt{2}C_0)^2.
\end{split}
\end{equation*}
Solving this inequality gives
\begin{equation*}
t  \leq  \frac{2}{\sqrt{u_{0,x}(x_0)^2-( u_0(x_0) - \frac{k}{2} )^2}-\sqrt{2}C_0} < \infty.
\end{equation*}
This in turn implies there exists $T < \infty$, such that
$$\liminf_{t \uparrow T_{u_0}, x \in \mathbb{R}} \partial_x u(t, x) = - \infty,$$
the desired result as indicated in Theorem \ref{Blow-up}.
\end{proof}
\begin{remark}
Returning to the original scale, our assumption for the blow-up phenomena becomes
\begin{equation*}
\sqrt{\beta \mu}\,u_{0,x}(\sqrt{\beta \mu}x_0) + \left |u_0(\sqrt{\beta \mu}x_0)- \frac {1}{2\alpha \varepsilon} \left ( \frac{\beta_0}{\beta} - c \right ) \right | < -\frac{\sqrt{2}}{\alpha \varepsilon} C_1.
\end{equation*}
Note that when  $\Omega$ increases, $\alpha$ and $\beta$ decrease. It is then observed that  with effect of the Earth rotation, a worse initial data $u_0(x_0)$ are required to make  the breaking wave happen.
On the other hand, in the original scale, we have
\begin{equation*}
T \leq \frac{2}{\alpha \varepsilon \left ( \sqrt{\beta \mu u_{0,x}^2(\sqrt{\beta\mu} x_0)-  \left (u_0(\sqrt{\beta \mu}x_0)- \frac{1}{2\alpha \varepsilon}  ( \frac{\beta_0}{\beta} - c  ) \right  )^2} -\frac{\sqrt{2}}{\alpha \varepsilon} C_1 \right ) }
\end{equation*}
where
\begin{equation*}
C_1^2 =  \frac{ |\omega_1|\alpha \varepsilon^3}{2}E^{\frac{3}{2}} + \frac{|\omega_2|\varepsilon^2 }{2 \alpha} E^2 \quad \mbox{with} \quad
E(u_0) = \frac{1}{\alpha^2 \varepsilon^2}E_0(\alpha \varepsilon u_0(\sqrt{\beta \mu}x_0)).
\end{equation*}

\end{remark}

\renewcommand{\theequation}{\thesection.\arabic{equation}}
\setcounter{equation}{0}


\appendix
\section{Derivations of the asymptotic expansions of $u, \,W, \,p,\,\eta$}

We consider the governing equations \eqref{governing}
\begin{equation}\label{A-Euler-1}
\begin{cases}
- c u_{\xi} + \varepsilon (u_\tau + uu_\xi + Wu_z) + 2\Omega W = - p_\xi \quad &  \text{in}\quad 0 < z < 1 + \varepsilon \eta,\\
\varepsilon\mu \{- c W_\xi + \varepsilon (W_\tau + u W_\xi + WW_z)\} - 2\Omega u = - p_z \quad &  \text{in}\quad 0 < z < 1 + \varepsilon \eta,\\
u_\xi + W_z = 0 \quad & \text{in}\quad 0 < z < 1 + \varepsilon \eta,\\
u_z - \varepsilon\mu W_\xi = 0 \quad & \text{in}\quad 0 < z < 1 + \varepsilon \eta,\\
p = \eta \quad & \text{on}\quad z = 1+ \varepsilon \eta,\\
W = - c \eta_\xi + \varepsilon (\eta_\tau + u \eta_\xi) \quad & \text{on}\quad z = 1+ \varepsilon \eta,\\
W = 0 \quad & \text{on} \quad z = 0.
\end{cases}
\end{equation}
A double asymptotic expansion is introduced to seek a solution of the system \eqref{A-Euler-1},
\begin{equation*}
q \sim \sum_{n=0}^{\infty} \sum_{m=0}^{\infty}\varepsilon^n \mu^m q_{nm}
\end{equation*}
as $\varepsilon \rightarrow 0, \mu \rightarrow 0$, where $q$ will be taken the scale functions $u, \,W, \,p$ and $\eta$, and all the functions $q_{nm}$ satisfiy the far field conditions $q_{nm} \rightarrow 0$ as $|\xi|\rightarrow \infty$ for every $n, \,m=0, 1, 2, 3, ...$.

Substituting the asymptotic expansions of $u, \,W, \,p,\,\eta$ into \eqref{A-Euler-1}, we check all the coefficients of the order $O(\varepsilon^i\mu^j)$ ($i, \, j=0, 1, 2, 3, ...$).

From the order $O(\varepsilon^0 \mu^0)$ terms of \eqref{A-Euler-1} we obtain
\begin{equation}\label{A-equation-00}
\begin{cases}
-c u_{00,\xi} + 2\Omega W_{00} = - p_{00,\xi} &\text{in}\quad 0 < z < 1,\\
2\Omega u_{00} = p_{00,z} &\text{in}\quad 0 < z < 1,\\
u_{00,\xi} + W_{00,z} = 0 &\text{in}\quad 0 < z < 1,\\
u_{00,z} = 0 &\text{in}\quad 0 < z < 1,\\
p_{00} = \eta_{00} & \text{on} \quad z = 1, \\
W_{00} = - c \eta_{00,\xi} & \text{on} \quad z = 1,\\
W_{00} = 0 & \text{on} \quad z = 0.
\end{cases}
\end{equation}
To solve the system \eqref{A-equation-00}, we first obtain from the fourth equation in \eqref{A-equation-00} that $u_{00}$ is independent of $z$, that is,
$u_{00} = u_{00}(\tau, \xi)$.

 Thanks to the third equation in \eqref{A-equation-00} and the boundary condition of $W$ on $z=0$, we get
\begin{equation} \label{A-w00-1}
W_{00} =W_{00}|_{z = 0} + \int_0^z W_{00,z'} dz' = -\int_0^z u_{00,\xi}\, dz'= - z u_{00,\xi},
\end{equation}
which along with the boundary condition of $W$ on $z=1$ implies
\begin{equation}\label{A-u-00-1}
u_{00,\xi}(\tau, \xi) = c\eta_{00,\xi}(\tau, \xi).
\end{equation}
Thereore, we have
\begin{equation}\label{A-w-00}
u_{00}(\tau, \xi) = c\eta_{00}(\tau, \xi), \quad W_{00} = -cz\eta_{00,\xi},
\end{equation}
here use has been made of the far field conditions $u_{00}, \, \eta_{00} \rightarrow 0$ as $|\xi| \rightarrow \infty$.

On the other hand, from the second equation in \eqref{A-equation-00}, there appears that
\begin{equation}\label{A-p-00-1}
p_{00}= p_{00}|_{z = 1} + \int_1^z p_{00,z'} \,dz'=\eta_{00}+2\Omega \int_1^z u_{00} \,dz'=\eta_{00}+2\Omega (z-1) u_{00},
\end{equation}
which along with \eqref{A-u-00-1} implies
\begin{equation}\label{A-p-00-2}
p_{00, \xi}=\big(\frac{1}{c}+2\Omega (z-1)\big) u_{00, \xi},
\end{equation}
Combining \eqref{A-p-00-2} with \eqref{A-w00-1} and the first equation in \eqref{A-equation-00} gives rise to
\begin{equation*}
(c^2 + 2\Omega c - 1) u_{00, \xi} = 0,
\end{equation*}
which follows that
\begin{equation}\label{A-c-1-2}
c^2 + 2\Omega c - 1= 0,
\end{equation}
if we assume that $u_{00}$ is an non-trivial velocity. Therefore, when consider the waves move towards to the right side, we may obtain
\begin{equation}\label{A-c-1-3}
c = \sqrt{1 + \Omega^2} - \Omega.
\end{equation}


Vanishing the order $O(\varepsilon^1 \mu^0)$ terms of \eqref{A-Euler-1}, we obtain from the second equation in \eqref{A-equation-10} and the Taylor expansion
\begin{equation}\label{A-taylor-1}
f(z)=f(1)+\sum_{n=1}^{\infty} \frac{(z-1)^n}{n!}f^{(n)}(1)
\end{equation}
that
\begin{equation}\label{A-equation-10}
\begin{cases}
 - c u_{10,\xi} + u_{00,\tau} + u_{00} u_{00,\xi} + 2\Omega W_{10} = - p_{10,\xi} &\text{in}\quad 0 < z < 1, \\
2\Omega u_{10} = p_{10,z}  &\text{in}\quad 0 < z < 1, \\
u_{10,\xi} + W_{10,z} = 0 &\text{in}\quad 0 < z < 1, \\
u_{10,z} = 0 &\text{in}\quad 0 < z < 1, \\
p_{10} + p_{00, z}\eta_{00} = \eta_{10}&\text{on}\quad z = 1,\\
W_{10}+\eta_{00} W_{00, z}= - c \eta_{10,\xi} + \eta_{00,\tau} + u_{00}\eta_{00,\xi} &\text{on}\quad z = 1,\\
W_{10} = 0&\text{on}\quad z = 0.
\end{cases}
\end{equation}
From the fourth equation in \eqref{A-equation-10}, we know that $u_{10}$ is independent to $z$, that is, $u_{10} = u_{10}(\tau, \xi)$. Thanks to the third equation in \eqref{A-equation-10} and the boundary conditions of $W$ on $z=0$ and $z=1$, we get
\begin{equation}\label{A-w-10-1}
W_{10} = W_{10}|_{z = 0} + \int_0^z W_{10,z'} dz'= - z u_{10,\xi}
\end{equation}
and
\begin{equation*}
\begin{split}
W_{10}|_{z = 1} =- c \eta_{10,\xi} + \eta_{00,\tau} + (u_{00}\eta_{00})_{\xi} . \end{split}
\end{equation*}
Hence, we obtain from the third equation in \eqref{equation-00} and \eqref{w-00} that
\begin{equation}\label{A-u-10-1}
u_{10,\xi} = c\eta_{10,\xi} - \eta_{00,\tau} - (u_{00}\eta_{00})_{\xi},
\end{equation}
and then
\begin{equation*}
W_{10} = z(\eta_{00,\tau}+2c\eta_{00}\eta_{00,\xi}-c\eta_{10,\xi} ).
\end{equation*}
On the other hand, thanks to the second equation in \eqref{A-equation-10} and \eqref{A-w-00}, we deduce that
\begin{equation*}
\begin{split}
p_{10} &= p_{10}|_{z = 1} + \int_1^z p_{10,z'} dz' = \eta_{10}-2\Omega u_{00}\eta_{00} +2\Omega (z-1)u_{10},\end{split}
\end{equation*}
and then
\begin{equation}\label{A-p10-1a}
\begin{split}
p_{10, \xi} &= \eta_{10, \xi}-2\Omega (u_{00}\eta_{00})_{\xi} +2\Omega (z-1)u_{10, \xi}.\end{split}
\end{equation}
Taking account of  the first equation in \eqref{A-equation-10} and \eqref{A-w-00}, it must be
\begin{equation*}
- p_{10,\xi}=- c u_{10,\xi} + c \eta_{00,\tau} + c^2 \eta_{00} \eta_{00,\xi} - 2\Omega z u_{10, \xi},
\end{equation*}
which along with \eqref{A-p10-1a} and \eqref{A-u-10-1} implies
\begin{equation*}
\begin{split}
0=&- (c+2\Omega) u_{10,\xi} +\eta_{10, \xi}+ c \eta_{00,\tau} + c^2 \eta_{00} \eta_{00,\xi} -2\Omega (u_{00}\eta_{00})_{\xi}\\
=&c(u_{00}\eta_{00})_{\xi} -(c^2 + 2\Omega c - 1)\eta_{10, \xi}+ 2(c+\Omega) \eta_{00,\tau} + c^2 \eta_{00} \eta_{00,\xi}.
\end{split}\end{equation*}
Hence, it follows from \eqref{A-w-00} and \eqref{A-c-1-2} that
\begin{equation} \label{A-eta-00-eqn}
2(\Omega + c) \eta_{00,\tau} + 3c^2 \eta_{00}\eta_{00,\xi} = 0.
\end{equation}
Defining
\begin{equation}\label{A-c-0-0}
c_1 \overset{\text{def}}{=} -\frac{3c^2}{4(\Omega + c)}=-\frac{3 c^3}{2 (c^2 + 1)},
\end{equation}
we may rewrite \eqref{A-eta-00-eqn} as
\begin{equation}\label{A-eta-00-tau}
\eta_{00,\tau} = c_1 (\eta_{00}^2)_{\xi},
\end{equation}
which, together with \eqref{A-u-10-1}, implies
\begin{equation}\label{A-u-10-xi}
u_{10,\xi} = \big(c \eta_{10}- (c + c_1) \eta_{00}^2\big)_\xi.
\end{equation}
Therefore, we get from the far field conditions $u_{10}, \, \eta_{00}, \eta_{10} \rightarrow 0$ as $|\xi| \rightarrow \infty$ that
\begin{equation}\label{A-u-10-3}
u_{10} =  c \eta_{10}- (c + c_1) \eta_{00}^2,
\end{equation}
which follows from \eqref{A-eta-00-tau} that
\begin{equation}\label{A-u-10-tau}
u_{10, \tau} =  c \eta_{10, \tau}- 4(c + c_1)c_1 \eta_{00}^2 \eta_{00, \xi}.
\end{equation}
Similarly, vanishing the order $O(\varepsilon^0 \mu^1)$ terms of \eqref{A-Euler-1}, we obtain from the second equation in \eqref{A-equation-00} and the Taylor expansion \eqref{taylor-1} that
\begin{equation*}
\begin{cases}
- c u_{01,\xi} + 2\Omega W_{01} = - p_{01,\xi} \quad &\text{in}\quad 0 < z < 1,\\
2\Omega u_{01} = p_{01,z} \quad &\text{in}\quad 0 < z < 1,\\
u_{01,\xi} + W_{01,z} = 0 \quad &\text{in}\quad 0 < z < 1,\\
u_{01,z} = 0 \quad & \text{in} \quad 0 < z < 1,\\
p_{01} = \eta_{01} \quad & \text{on} \quad z = 1,\\
W_{01} = - c \eta_{01,\xi} \quad & \text{on} \quad z = 1,\\
W_{01} = 0 \quad & \text{on} \quad z = 0.
\end{cases}
\end{equation*}
From this, we may readily get from the above argument that
\begin{equation}\label{A-uwp-01-1}
u_{01} = c\eta_{01}=c\eta_{01}(\tau, \xi), \, W_{01} = -cz\eta_{01,\xi},\, p_{01} = [2\Omega c(z-1)+1]\eta_{01}.
\end{equation}

For the order $O(\varepsilon^2 \mu^0)$ terms of \eqref{A-Euler-1}, we obtain from the Taylor expansion \eqref{A-taylor-1} that
\begin{equation}\label{A-equation-20}
\begin{cases}
 -c u_{20,\xi} + u_{10,\tau} + (u_{00}u_{10})_{\xi}  + 2\Omega W_{20} = -p_{20,\xi} \quad &\text{in}\quad 0 < z < 1, \\
 -2 \Omega u_{20}=-p_{20,z}  \quad&\text{in}\quad 0 < z < 1,\\
u_{20,\xi} + W_{20,z} = 0 \quad&\text{in}\quad 0 < z < 1, \\
u_{20,z} = 0 \quad &\text{in}\quad 0 < z < 1, \\
p_{20} + \eta_{00}p_{10,z} + \eta_{10}p_{00,z}= \eta_{20}\quad&\text{on}\quad z = 1,\\
W_{20} + \eta_{00}W_{10,z} + \eta_{10}W_{00,z}  \\
\quad\quad\quad= - c \eta_{20,\xi} + \eta_{10,\tau} + u_{00}\eta_{10,\xi} + u_{10} \eta_{00,\xi} \quad &\text{on}\quad z = 1,\\
W_{20} = 0\quad&\text{on}\quad z = 0.
\end{cases}
\end{equation}
From the fourth equation in \eqref{A-equation-20}, we know that $u_{20}$ is independent of $z$, that is,
$u_{20}=u_{20}(\tau, \xi)$,
which along with the third equation in \eqref{A-equation-20} and the boundary condition of $W_{20}$ at $z=0$ implies that
\begin{equation}\label{A-w-20-1}
W_{20}=-z u_{20, \xi}.
\end{equation}
Combining \eqref{A-w-20-1} with  the boundary condition of $W_{20}$ at $z=1$, we get from the equations of $W_{00, z}$ and $W_{10, z}$ that
\begin{equation*}
u_{20, \xi}= c \eta_{20,\xi} - \eta_{10,\tau} - (u_{00}\eta_{10} + u_{10} \eta_{00})_{\xi},
\end{equation*}
that is,
\begin{equation}\label{A-w-20-3}
u_{20, \xi}= c \eta_{20,\xi} - \eta_{10,\tau} - 2c(\eta_{00}\eta_{10})_{\xi} +(c+c_1)(\eta_{00}^3)_{\xi}.
\end{equation}
While from the second equation in \eqref{A-equation-20} and the boundary condition of $p_{20}$ at $z=1$, we get
\begin{equation*}
\begin{split}
p_{20}=p_{20}|_{z=1}+\int_1^z p_{20, z'}\, dz'&= \eta_{20}-(\eta_{00}p_{10,z} + \eta_{10}p_{00,z} )+2\Omega \int_1^z u_{20}\, dz'\\
&= \eta_{20}-2\Omega (\eta_{00}u_{10} + \eta_{10}u_{00} )+2\Omega (z-1) u_{20},
\end{split}\end{equation*}
which leads to
\begin{equation}\label{A-p-20-2}
\begin{split}
p_{20, \xi}&= \eta_{20, \xi}-2\Omega (\eta_{00}u_{10} + \eta_{10}u_{00} )_{\xi}+2\Omega (z-1) u_{20, \xi}.
\end{split}\end{equation}
On the other hand, due to the first equation in \eqref{A-equation-20}, we deduce from \eqref{A-w-20-1} and \eqref{A-w-20-3} that
\begin{equation}\label{A-p-20-3}
-p_{20,\xi}= -c u_{20,\xi} + u_{10,\tau} + (u_{00}u_{10})_{\xi}  - 2\Omega z u_{20, \xi}.
\end{equation}
Combining \eqref{A-p-20-2} with \eqref{A-p-20-3}, we have
\begin{equation*}
\eta_{20, \xi}-2\Omega (\eta_{00}u_{10} + \eta_{10}u_{00} )_{\xi}-(c+2\Omega) u_{20,\xi} + u_{10,\tau} + (u_{00}u_{10})_{\xi} =0.
\end{equation*}
Thanks to \eqref{A-u-00-1}, \eqref{A-u-10-3}, and \eqref{A-u-10-tau}, we obtain
\begin{equation}\label{A-eta-10-eqn}
\begin{split}
2(c+\Omega) \eta_{10,\tau}+3c^2(\eta_{00}\eta_{10})_{\xi}  -(2c+\frac{4}{3}c_1)(c + c_1)(\eta_{00}^3)_{\xi}  =0,
\end{split}\end{equation}
which leads to
\begin{equation}\label{A-eta-10-tau}
 \eta_{10, \tau}= 2c_1(\eta_{00}\eta_{10})_{\xi}+\frac{2c_1+3c}{3(c+\Omega)}(c+c_1)(\eta_{00}^3)_{\xi}.
\end{equation}
Therefore, we have
\begin{equation*}
u_{20, \xi}= c \eta_{20,\xi} - 2(c+c_1)(\eta_{00}\eta_{10})_{\xi}-\frac{2c_1-3\Omega}{3(c+\Omega)}(c+c_1)(\eta_{00}^3)_{\xi},
\end{equation*}
which along with the far field conditions $\eta_{00},\, \eta_{10}, \,\eta_{20}\rightarrow 0$ as $|\xi| \rightarrow \infty$ gives
\begin{equation}\label{A-u-20}
u_{20}= c \eta_{20} - 2(c+c_1)\eta_{00}\eta_{10}-\frac{2c_1-3\Omega}{3(c+\Omega)}(c+c_1)\eta_{00}^3.
\end{equation}
Thanks to \eqref{A-eta-00-tau} and \eqref{A-eta-10-tau}, we deduce that
\begin{equation}\label{A-u-20-tau}
\begin{split}
u_{20, \tau}=c \eta_{20, \tau} - 4(c+c_1)c_1 (\eta_{00}^2\eta_{10})_{\xi}-\frac{8cc_1+4c_1^2+\frac{21}{4}c^2}{2(c+\Omega)}(c+c_1) (\eta_{00}^4)_{\xi}.
\end{split}\end{equation}

For the order $O(\varepsilon^1 \mu^1)$ terms of \eqref{A-Euler-1}, we obtain from the Taylor expansion \eqref{A-taylor-1} that
\begin{equation}\label{A-equation-11}
\begin{cases}
 - c u_{11,\xi} + u_{01,\tau} + u_{00}u_{01, \xi}+u_{10}u_{00, \xi}+ W_{00}u_{01, z}\\
 \qquad\qquad\qquad\qquad\qquad+W_{10}u_{00, z}+ 2\Omega W_{11} =  - p_{11,\xi} \quad &\text{in}\quad 0 < z < 1, \\
-cW_{00,\xi} - 2 \Omega u_{11} = - p_{11,z}  \quad&\text{in}\quad 0 < z < 1,\\
u_{11,\xi} + W_{11,z} = 0 \quad&\text{in}\quad 0 < z < 1, \\
u_{11,z} - W_{00,\xi}= 0 \quad &\text{in}\quad 0 < z < 1, \\
p_{11} = \eta_{11}-(\eta_{00}p_{01, z} +\eta_{01}p_{00, z})\quad&\text{on}\quad z = 1,\\
W_{11} +W_{00, z} \eta_{01}+W_{01, z} \eta_{00} \\
\qquad\qquad = - c \eta_{11,\xi}+\eta_{01,\tau} + u_{00}\eta_{01, \xi}+ u_{01}\eta_{00, \xi} \quad &\text{on}\quad z = 1,\\
W_{11} = 0\quad&\text{on}\quad z = 0.
\end{cases}
\end{equation}
Thanks to \eqref{A-w-00} and the fourth equation of \eqref{A-equation-11}, we have $u_{11, z} = -c z\eta_{00, \xi\xi}$,
and then
\begin{equation}\label{A-u-11-2}
u_{11} = -\frac{c}{2}z^2 \eta_{00, \xi\xi}+ \Phi_{11}(\tau, \xi)
\end{equation}
for some arbitrary smooth function $\Phi_{11}(\tau, \xi)$ independent of $z$.
While from the third equation in  \eqref{A-equation-11} with $W_{11}|_{z=0} = 0$, it follows that
\begin{equation}\label{A-w-11-1}
W_{11} = W_{11}|_{z=0} +\int_0^z W_{11, z'}\,dz'=\frac{c}{6}z^3 \eta_{00, \xi\xi\xi}- z\partial_{\xi}\Phi_{11}(\tau, \xi),
\end{equation}
which, along with the equations of $W_{00, z}$ and $W_{01, z}$, and the boundary condition of $W_{11}$ on $\{z=1\}$, implies
\begin{equation}\label{A-w-11-2}
\begin{split}
- \partial_{\xi}\Phi_{11}(\tau, \xi)=
-\frac{c}{6} \eta_{00, \xi\xi\xi}+ (u_{00}\eta_{01}+\eta_{00}u_{01})_{\xi}- c \eta_{11,\xi}+\eta_{01,\tau}.
\end{split}
\end{equation}
Hence, in view of \eqref{A-w-11-1}, \eqref{A-w-00}, \eqref{A-uwp-01-1}, and \eqref{A-u-00-1}, we obtain
\begin{equation}\label{A-w-11-3}
W_{11} =\frac{c}{6}z(z^2-1) \eta_{00, \xi\xi\xi}+z \bigg(- c \eta_{11,\xi}+\eta_{01,\tau} + (u_{00}\eta_{01}+\eta_{00}u_{01})_{\xi}\bigg).
\end{equation}
Due to \eqref{A-w-00}, \eqref{A-uwp-01-1}, \eqref{A-u-11-2}, and the boundary condition of $p_{11}$ in \eqref{A-equation-11},  we deduce from the second equation of \eqref{A-equation-11} that
\begin{equation*}
\begin{split}
&p_{11} =p_{11}|_{z = 1} + \int_1^z p_{11,z'}\, dz'=p_{11}|_{z = 1} + \int_1^z (cW_{00,\xi} +2 \Omega u_{11}) \, dz'\\
&=\eta_{11} - 2\Omega (u_{00}\eta_{01}+\eta_{00}u_{01}) - \bigg(\frac{c^2 }{2}(z^2-1)+\frac{\Omega c}{3}(z^3-1)\bigg)\eta_{00,\xi\xi}+2\Omega (z-1)\Phi_{11},
\end{split}
\end{equation*}
which implies
\begin{equation}\label{A-p-11-2}
\begin{split}
p_{11, \xi} =\eta_{11, \xi} - 2\Omega (u_{00}\eta_{01}+\eta_{00}u_{01})_{\xi} &- \bigg(\frac{c^2}{2}(z^2-1)+\frac{\Omega c}{3}(z^3-1)\bigg)\eta_{00,\xi\xi\xi}\\
&+2\Omega (z-1)\partial_{\xi}\Phi_{11}.
\end{split}
\end{equation}
Combining \eqref{A-p-11-2} and the first equation in \eqref{A-equation-11}, it follows  from \eqref{A-w-00}, \eqref{A-uwp-01-1}, and \eqref{A-u-11-2} that
\begin{equation}\label{A-p-11-3}
\begin{split}
&- c u_{11,\xi} + c\eta_{01,\tau} + c^2(\eta_{00}\eta_{01})_{\xi} + 2\Omega W_{11}+\eta_{11, \xi} - 4\Omega c (\eta_{00}\eta_{01})_{\xi}\\
& - \bigg(\frac{c^2}{2}(z^2-1)+\frac{\Omega c}{3}(z^3-1)\bigg)\eta_{00,\xi\xi\xi}+2\Omega (z-1)\partial_{\xi}\Phi_{11}=0.
\end{split}
\end{equation}
Substituting \eqref{A-u-11-2} and \eqref{A-w-11-2} into \eqref{A-p-11-3}, we obtain
\begin{equation}\label{A-eta-01-eqn}
\begin{split}
&2(\Omega+c) \eta_{01,\tau} + 3c^2(\eta_{00}\eta_{01})_{\xi} + \frac{c^2}{3}\eta_{00,\xi\xi\xi}=0,
\end{split}
\end{equation}
that is,
\begin{equation} \label{A-eta-01-tau}
\eta_{01,\tau} = 2 c_1 (\eta_{00}\eta_{01})_\xi + \frac{2 c_1}{9} \eta_{00,\xi\xi\xi},
\end{equation}
which, together with \eqref{A-w-11-2}, \eqref{A-w-11-3}, and \eqref{A-u-11-2}, leads to
\begin{equation*}
\begin{split}
- \partial_{\xi}\Phi_{11}(\tau, \xi)=(\frac{2 c_1}{9}-\frac{c}{6}) \eta_{00, \xi\xi\xi}+ 2(c +c_1)(\eta_{00}\eta_{01})_{\xi}- c \eta_{11,\xi},
\end{split}
\end{equation*}
and then
\begin{equation*}
W_{11} =\bigg(\frac{2 c_1}{9}+\frac{c}{6}(z^2-1)\bigg)\, z\,\eta_{00, \xi\xi\xi}+ 2(c +c_1)\,z\,(\eta_{00}\eta_{01})_{\xi}- c\,z\, \eta_{11,\xi}
\end{equation*}
and
\begin{equation}\label{A-u-11}
u_{11} = \left(\frac{c}{6} - \frac{2c_1}{9} -\frac{c}{2}z^2 \right) \eta_{00,\xi\xi} + c \eta_{11} - 2 (c+c_1) \eta_{00}\eta_{01},
\end{equation}
where use has been made by  the far field conditions $u_{11}, \, \eta_{00,\xi\xi},\, \eta_{00}, \,\eta_{01},\, \eta_{11}\rightarrow 0$ as $|\xi| \rightarrow \infty$.

Thanks to \eqref{A-eta-00-tau} and \eqref{A-eta-01-tau}, we obtain
\begin{equation}\label{A-u-11-tau}
\begin{split}
u_{11, \tau} =&c \eta_{11, \tau} +\left(\frac{cc_1 }{6} - \frac{2c_1^2}{9} -\frac{cc_1 }{2}z^2 \right)(\eta_{00}^2)_{\xi\xi\xi}  \\
&- 2 (c+c_1) \bigg(2 c_1 (\eta_{00}^2\eta_{01})_\xi + \frac{2 c_1}{9} \eta_{00}\eta_{00,\xi\xi\xi}\bigg).
\end{split}
\end{equation}


For the order $O(\varepsilon^3 \mu^0)$ terms of \eqref{A-Euler-1}, we obtain from the Taylor expansion \eqref{A-taylor-1} that
\begin{equation}\label{A-equation-30}
\begin{cases}
 -c u_{30,\xi} + u_{20,\tau} + (u_{00}u_{20}+\frac{1}{2}u_{10}^2)_{\xi}  + 2\Omega W_{30} = -p_{30,\xi} \quad &\text{in}\quad 0 < z < 1, \\
 -2 \Omega u_{30}=-p_{30,z}  \quad &\text{in}\quad 0 < z < 1,\\
u_{30,\xi} + W_{30,z} = 0 \quad &\text{in}\quad 0 < z < 1, \\
u_{30,z} = 0 \quad &\text{in}\quad 0 < z < 1, \\
p_{30} + \eta_{00}p_{20,z} + \eta_{10}p_{10,z} + \eta_{20}p_{00,z}= \eta_{30}\quad &\text{on}\quad z = 1,\\
W_{30} + \eta_{00}W_{20,z} + \eta_{10}W_{10,z} + \eta_{20}W_{00,z}  \\
\quad\quad\quad= - c \eta_{30,\xi} + \eta_{20,\tau} + u_{00}\eta_{20,\xi} + u_{10} \eta_{10,\xi} + u_{20} \eta_{00,\xi} \quad &\text{on}\quad z = 1,\\
W_{30} = 0\quad &\text{on}\quad z = 0.
\end{cases}
\end{equation}
From the fourth equation in \eqref{A-equation-30}, we know that $u_{30}$ is independent of $z$, that is,
$u_{30}=u_{30}(\tau, \xi)$,
which along with the third equation in \eqref{A-equation-30} and the boundary condition of $W_{30}$ at $z=0$ implies that
$W_{30}=-z u_{30, \xi}$.
Combining \eqref{w-20-1} with  the boundary condition of $W_{20}$ at $z=1$, we have
\begin{equation}\label{A-u-30-xi-1}
u_{30, \xi}= c \eta_{30,\xi} - \eta_{20,\tau} - (u_{00}\eta_{20} + u_{10} \eta_{10}+ u_{20} \eta_{00})_{\xi}.
\end{equation}
While from the second equation in \eqref{A-equation-30} and the boundary condition of $p_{30}$ at $z=1$, we get
\begin{equation*}
\begin{split}
&p_{30}=p_{30}|_{z=1}+\int_1^z p_{30, z'}\, dz'\\
&= \eta_{30}-(\eta_{00}p_{20,z} + \eta_{10}p_{10,z} + \eta_{20}p_{00,z})+2\Omega \int_1^z u_{30}\, dz'\\
&= \eta_{30}-2\Omega (u_{00}\eta_{20} + u_{10} \eta_{10}+ u_{20} \eta_{00})+2\Omega (z-1) u_{30},
\end{split}\end{equation*}
which leads to
\begin{equation}\label{A-p-30-xi-1}
\begin{split}
p_{30, \xi}&= \eta_{30, \xi}-2\Omega (u_{00}\eta_{20} + u_{10} \eta_{10}+ u_{20} \eta_{00})_{\xi}+2\Omega (z-1) u_{30, \xi}.
\end{split}\end{equation}
On the other hand, from the first equation in \eqref{A-equation-30}, we have
\begin{equation}\label{A-p-30-xi-2}
\begin{split}
 -p_{30,\xi}=-c u_{30,\xi} + u_{20,\tau} + (u_{00}u_{20}+\frac{1}{2}u_{10}^2)_{\xi}  - 2\Omega z u_{30, \xi}.
\end{split}\end{equation}
Combining \eqref{A-p-30-xi-1} with \eqref{A-p-30-xi-2}, we get
\begin{equation}\label{A-p-30-xi-3}
\begin{split}
0= \eta_{30, \xi}-2\Omega (u_{00}\eta_{20} + u_{10} \eta_{10}+ u_{20} \eta_{00})_{\xi}-(c+2\Omega) u_{30, \xi}+ u_{20,\tau} + (u_{00}u_{20}+\frac{1}{2}u_{10}^2)_{\xi} .
\end{split}\end{equation}
Substituting \eqref{A-u-30-xi-1} and \eqref{A-u-20-tau} into \eqref{A-p-30-xi-3}, we obtain
\begin{equation}\label{A-eta-20-eqn}
\begin{split}
2(c+\Omega)\eta_{20,\tau} +3c^2(\eta_{00}\eta_{20})_{\xi} &+ \frac{3c^2}{2}(\eta_{10}^2)_{\xi}-2(2c_1+3c)(c+c_1)(\eta_{00}^2\eta_{10})_{\xi}\\
&-\frac{(64cc_1+24c_1^2+45c^2-15)}{12(c+\Omega)}(c+c_1)(\eta_{00}^4)_{\xi}=0,
\end{split}\end{equation}
that is,
\begin{equation}\label{A-eta-20-tau}
\begin{split}
\eta_{20,\tau}= 2c_1(\eta_{00}\eta_{20})_{\xi} +c_1(\eta_{10}^2)_{\xi}&+\frac{2c_1+3c}{\Omega+c}(c+c_1)(\eta_{00}^2\eta_{10})_{\xi}\\
&+\frac{(64cc_1+24c_1^2+45c^2-15)}{24(c+\Omega)^2}(c+c_1)(\eta_{00}^4)_{\xi}.
\end{split}\end{equation}
Thanks to \eqref{A-u-30-xi-1} again, we have
\begin{equation*}
\begin{split}
u_{30, \xi}=& c \eta_{30,\xi} -  2(c+c_1)(\eta_{00}\eta_{20})_{\xi} -(c+c_1)(\eta_{10}^2)_{\xi}-\frac{2c_1-3\Omega}{\Omega+c}(c+c_1)(\eta_{00}^2\eta_{10})_{\xi}\\
&-\frac{(64cc_1+24c_1^2+45c^2+24\Omega^2-3)}{24(c+\Omega)^2}(c+c_1)(\eta_{00}^4)_{\xi},
\end{split}\end{equation*}
which implies
\begin{equation}\label{A-u-30}
\begin{split}
u_{30}=& c \eta_{30} -  2(c+c_1)(\eta_{00}\eta_{20}) -(c+c_1)(\eta_{10}^2)-\frac{2c_1-3\Omega}{\Omega+c}(c+c_1)(\eta_{00}^2\eta_{10})\\
&-\frac{(64cc_1+24c_1^2+45c^2+24\Omega^2-3)}{24(c+\Omega)^2}(c+c_1)(\eta_{00}^4).
\end{split}\end{equation}
Therefore, due to \eqref{A-eta-00-tau}, \eqref{A-eta-10-tau}, and \eqref{A-eta-20-tau}, we have
\begin{equation}\label{A-u-30-tau}
\begin{split}
u_{30, \tau}=& c \eta_{30, \tau} - \frac{2(3c^2+5cc_1+4c_1^2-3\Omega c_1)}{\Omega+c}(c+c_1)(\eta_{00}^3\eta_{10})_{\xi} \\
&-4c_1(c+c_1)(\eta_{00}\eta_{10}^2)_{\xi}-4c_1(c+c_1)(\eta_{00}^2\eta_{20})_{\xi}
-B_1\eta_{00}^4\eta_{00, \xi}
\end{split}\end{equation}
with
\begin{equation*}\begin{split}
B_1\overset{\text{def}}{=} &\frac{(c+c_1)^2(82cc_1+36c_1^2+45c^2-18\Omega c_1-27\Omega c-15)}{3(\Omega+c)^2}\\
&+\frac{c_1(c+c_1)(64cc_1+24c_1^2+45c^2+24\Omega^2-3)}{3(\Omega+c)^2}.
\end{split}\end{equation*}
For the terms of \eqref{A-Euler-1} at order $O(\varepsilon^4 \mu^0)$, it is inferred  from the Taylor expansion \eqref{A-taylor-1} that
\begin{equation}\label{A-equation-40}
\begin{cases}
 -c u_{40,\xi} + u_{30,\tau} + (u_{00}u_{30}+u_{10}u_{20})_{\xi}  + 2\Omega W_{40} = -p_{40,\xi} \quad &\text{in}\quad 0 < z < 1, \\
 -2 \Omega u_{40}=-p_{40,z}  \quad&\text{in}\quad 0 < z < 1,\\
u_{40,\xi} + W_{40,z} = 0 \quad&\text{in}\quad 0 < z < 1, \\
u_{40,z} = 0 \quad &\text{in}\quad 0 < z < 1, \\
p_{40} + \eta_{00}p_{30,z} + \eta_{10}p_{20,z} + \eta_{20}p_{10,z}+ \eta_{30}p_{00,z}= \eta_{40}\quad&\text{on}\quad z = 1,\\
W_{40} + \eta_{00}W_{30,z} + \eta_{10}W_{20,z} + \eta_{20}W_{10,z}+ \eta_{30}W_{00,z}   \\
\quad= - c \eta_{40,\xi} + \eta_{30,\tau} + u_{00}\eta_{30,\xi} + u_{10} \eta_{20,\xi} + u_{20} \eta_{10,\xi}+ u_{30} \eta_{00,\xi} \quad &\text{on}\quad z = 1,\\
W_{40} = 0\quad&\text{on}\quad z = 0.
\end{cases}
\end{equation}

From the fourth equation in \eqref{A-equation-30}, we know that $u_{40}$ is independent of $z$, that is,
$u_{40}=u_{40}(\tau, \xi)$,
which along with the third equation in \eqref{A-equation-40} and the boundary condition of $W_{40}$ at $z=0$ implies that
\begin{equation}\label{A-w-40-1}
W_{40}=-z u_{40, \xi}.
\end{equation}
Combining \eqref{A-w-40-1} with  the boundary condition of $W_{40}$ at $z=1$, we have
\begin{equation}\label{A-u-40-xi-1}
u_{40, \xi}= c \eta_{40,\xi} - \eta_{30,\tau} - (u_{00}\eta_{30} + u_{10} \eta_{20}+ u_{20} \eta_{10}+ u_{30} \eta_{00})_{\xi},
\end{equation}
From the second equation in \eqref{A-equation-40} and the boundary condition of $p_{30}$ at $z=1$, we get
\begin{equation*}
\begin{split}
&p_{40}=p_{40}|_{z=1}+\int_1^z p_{40, z'}\, dz'\\
&= \eta_{40}-(\eta_{00}p_{30,z} + \eta_{10}p_{20,z} + \eta_{20}p_{10,z}+ \eta_{30}p_{00,z})+2\Omega \int_1^z u_{40}\, dz'\\
&= \eta_{40}-2\Omega (u_{00}\eta_{30} + u_{10} \eta_{20}+ u_{20} \eta_{10}+ u_{30} \eta_{00})+2\Omega (z-1) u_{40},
\end{split}\end{equation*}
which implies
\begin{equation}\label{A-p-40-xi-1}
\begin{split}
p_{40, \xi}&= -\eta_{40, \xi}-2\Omega (u_{00}\eta_{30} + u_{10} \eta_{20}+ u_{20} \eta_{10}+ u_{30} \eta_{00})_{\xi}+2\Omega (z-1) u_{40, \xi}.
\end{split}\end{equation}
On the other hand, from the first equation in \eqref{A-equation-40}, we have
\begin{equation*}
\begin{split}
 -p_{40,\xi} =-c u_{40,\xi} + u_{30,\tau} + (u_{00}u_{30}+u_{10}u_{20})_{\xi}  + 2\Omega W_{40},
\end{split}\end{equation*}
which along with \eqref{A-w-40-1} and \eqref{A-p-40-xi-1} gives rise to
\begin{equation}\label{A-p-40-xi-3}
\begin{split}
0=&-(c+2\Omega) u_{40,\xi} + u_{30,\tau} + (u_{00}u_{30}+u_{10}u_{20})_{\xi} \\
&+\eta_{40, \xi}-2\Omega (u_{00}\eta_{30} + u_{10} \eta_{20}+ u_{20} \eta_{10}+ u_{30} \eta_{00})_{\xi}
\end{split}\end{equation}
Substituting \eqref{A-u-40-xi-1} and \eqref{A-u-30-tau} into \eqref{A-p-40-xi-3}, we obtain
\begin{equation}\label{A-eta-30-eqn}
\begin{split}
&2(c+\Omega)\eta_{30,\tau} +3c^2(\eta_{00}\eta_{30}+\eta_{10}\eta_{20})_{\xi} -2(3c+2c_1)(c+c_1)(\eta_{00}^2\eta_{20}+\eta_{00}\eta_{10}^2)_{\xi}\\
&\quad-\frac{(64cc_1+24c_1^2+45c^2-15)}{3(c+\Omega)}(c+c_1)(\eta_{00}^3\eta_{10})_{\xi}-B_2(\eta_{00}^5)_{\xi}=0
\end{split}\end{equation}
with
\begin{equation*}\begin{split}
B_2&\overset{\text{def}}{=} \frac{1}{5}B_1-\frac{(c+c_1)^2(2c_1-3\Omega)}{3(\Omega+c)}+\frac{2c(c+c_1)(64cc_1+24c_1^2+45c^2+24\Omega^2-3)}{12(\Omega+c)^2}\\
&=\frac{c^2(2-c^2)(3c^{10}+228c^8-540c^6-180c^4-13c^2+42)}{60(c^2+1)^6}.
\end{split}\end{equation*}


For the terms in \eqref{A-Euler-1} at  order $O(\varepsilon^2 \mu^1)$, we have
\begin{equation}\label{A-equation-21}
\begin{cases}
-c u_{21,\xi} + u_{11,\tau} + (u_{00}u_{11}+u_{10}u_{01})_\xi+ W_{00}u_{11,z}+2\Omega W_{21} = - p_{21,\xi}\quad &\text{in}\quad 0 < z < 1,\\
-cW_{10,\xi} + W_{00,\tau} + u_{00}W_{00,\xi} + W_{00}W_{00,z} - 2\Omega u_{21} = - p_{21,z} \quad &\text{in}\quad 0 < z < 1,\\
u_{21,\xi}+W_{21,z} = 0\quad &\text{in}\quad 0 < z < 1,\\
u_{21,z}-W_{10,\xi} = 0\quad &\text{in}\quad 0 < z < 1,\\
p_{21} + \eta_{10}p_{01,z} + \eta_{01}p_{10,z}+\eta_{00}p_{11,z}+\eta_{11}p_{00,z}= \eta_{21}\quad &\text{on}\quad z=1,\\
W_{21} + \eta_{10}W_{01,z}+\eta_{01}W_{10,z}+\eta_{00}W_{11,z}+\eta_{11}W_{00,z} \\
= - c \eta_{21,\xi} + \eta_{11,\tau}+u_{00}\eta_{11,\xi}+u_{11}\eta_{00,\xi} +u_{10}\eta_{01,\xi}+u_{01}\eta_{10,\xi} \quad &\text{on}\quad z=1,\\
W_{21} = 0\quad &\text{on}\quad z=0.
\end{cases}
\end{equation}
We now first derive from \eqref{A-w-10-1}, \eqref{A-u-10-xi}, and the fourth equation in \eqref{A-equation-21} that
\begin{equation*}
u_{21,z}=W_{10,\xi} = z\bigg(2(c+c_1)(\eta_{00,\xi}^2+\eta_{00}\eta_{00,\xi\xi})-c\eta_{10,\xi\xi}\bigg),
\end{equation*}
which gives
\begin{equation*}
u_{21}= \frac{z^2}{2}\bigg(2(c+c_1)(\eta_{00,\xi}^2+\eta_{00}\eta_{00,\xi\xi})-c\eta_{10,\xi\xi}\bigg)+ \Phi_{21}(\tau, \xi)=\frac{z^2}{2}H_1+ \Phi_{21}(\tau, \xi)
\end{equation*}
for some smooth function $\Phi_{21}(\tau, \xi)$ independent of $z$, where we denote
\begin{equation*}
H_1 \overset{\text{def}}{=}2(c+c_1)(\eta_{00,\xi}^2+\eta_{00}\eta_{00,\xi\xi})-c\eta_{10,\xi\xi}.
\end{equation*}
 Hence, we have
\begin{equation*}
\begin{split}
u_{21, \xi}= \frac{z^2}{2}H_{1, \xi}+ \partial_{\xi}\Phi_{21}(\tau, \xi).
\end{split}
\end{equation*}
On the other hand, thanks to the third equation in \eqref{A-equation-21} and the boundary condition of $W_{21}$ on $\{z=0\}$, we get
\begin{equation*}
\begin{split}
W_{21}&=W_{21}|_{z=0}+\int_0^z W_{21, z'}\,dz'=-\int_0^z u_{21, \xi}\,dz'= -\frac{z^3}{6}H_{1, \xi}-z \partial_{\xi}\Phi_{21}(\tau, \xi),
\end{split}
\end{equation*}
which along with the boundary condition of $W_{21}$ on $\{z=1\}$ leads to
\begin{equation*}
\begin{split}
-\frac{1}{6}H_{1, \xi}- \partial_{\xi}\Phi_{21}(\tau, \xi)&= - c \eta_{21,\xi} + \eta_{11,\tau}+(u_{00}\eta_{11}+u_{11}\eta_{00} +u_{10}\eta_{01}+u_{01}\eta_{10})_{\xi}|_{z=1}\\
&= - c \eta_{21,\xi} + \eta_{11,\tau}+H_{2, \xi}|_{z=1},
 \end{split}
\end{equation*}
where we denote
\begin{equation*}
H_2 \overset{\text{def}}{=} u_{00}\eta_{11}+u_{11}\eta_{00} +u_{10}\eta_{01}+u_{01}\eta_{10}.
\end{equation*}
It then follows that
\begin{equation}\label{A-phi-21-1}
\begin{split}
\partial_{\xi}\Phi_{21}(\tau, \xi)= c \eta_{21,\xi} - \eta_{11,\tau}-\frac{1}{6}H_{1, \xi}-H_{2, \xi}|_{z=1},  \end{split}
\end{equation}
which implies
\begin{equation}\label{A-u-21-xi-2}
\begin{split}
u_{21, \xi}=  c \eta_{21,\xi} - \eta_{11,\tau}+(\frac{z^2}{2}-\frac{1}{6})H_{1, \xi}-H_{2, \xi}|_{z=1}
\end{split}
\end{equation}
and
\begin{equation}\label{A-w-21-1}
\begin{split}
W_{21}= \frac{z(1-z^2)}{6}H_{1, \xi}-c z\eta_{21,\xi} +z \eta_{11,\tau}+z (H_{2, \xi}|_{z=1}).
\end{split}
\end{equation}
Substituting the expressions of $W_{00,\tau}$, $u_{00}$, $W_{00,\xi}$, $W_{00}$, $W_{00,z}$, and $W_{10,\xi}$ into the second equation in \eqref{A-equation-21}, we obtain
\begin{equation}\label{A-p-21-1}
p_{21,z} = 2\Omega u_{21}-c^2 z \eta_{10,\xi\xi}+c(c  +4  c_1 )z\eta_{00,\xi}^2+c(3c+4c_1 )z\eta_{00}\eta_{00,\xi\xi}.
\end{equation}
While from the boundary condition of $p_{21}$ on $z=1$, we have
\begin{equation*}
p_{21}|_{z=1} = \eta_{21}+c^2\eta_{00}\eta_{00, \xi\xi}-2\Omega H_{2}|_{z=1},
\end{equation*}
which along with \eqref{A-p-21-1} leads to
\begin{equation}\label{A-p-21-3}
\begin{split}
&p_{21}=p_{21}|_{z=1}+\int_1^z p_{21,z'}\,dz' \\
&= \eta_{21}-2\Omega H_{2}|_{z=1}+ 2\Omega \int_1^z u_{21}\, dz'-\frac{c^2}{2}(z^2-1) \eta_{10,\xi\xi}\\
&\quad+\frac{c(c +4c_1 )}{2} (z^2-1)\eta_{00,\xi}^2+\bigg(c^2+\frac{c(3c+4c_1)}{2} (z^2-1)\bigg)\eta_{00}\eta_{00,\xi\xi},
\end{split}
\end{equation}
and then
\begin{equation}\label{A-p-21-3a}
\begin{split}
&p_{21, \xi}= \eta_{21, \xi}-2\Omega H_{2, \xi}|_{z=1}+ 2\Omega \int_1^z u_{21, \xi}\, dz'-\frac{c^2 }{2}(z^2-1) \eta_{10,\xi\xi\xi}\\
&\quad+\frac{c(c  +4  c_1 )}{2}(z^2-1) (\eta_{00,\xi}^2)_{\xi}+\bigg(c^2+\frac{c(3c+4c_1 )}{2} (z^2-1)\bigg)(\eta_{00}\eta_{00,\xi\xi})_{\xi}\\
&= -2\Omega z H_{2, \xi}|_{z=1}+ 2\Omega (z-1) \bigg(c \eta_{21,\xi} - \eta_{11,\tau}\bigg)+\frac{z(z^2-1)}{6}H_{1, \xi}-\frac{c^2 }{2}(z^2-1) \eta_{10,\xi\xi\xi}\\
&\quad+\eta_{21, \xi}+\frac{c(c  +4  c_1 )}{2} (z^2-1)(\eta_{00,\xi}^2)_{\xi}+\bigg(c^2+\frac{c(3c+4c_1 )}{2}(z^2-1) \bigg)(\eta_{00}\eta_{00,\xi\xi})_{\xi}.
\end{split}
\end{equation}
Thanks to the first equation in \eqref{A-equation-21}, \eqref{A-w-21-1}, and \eqref{A-w-00}, we get
\begin{equation}\label{A-p-21-4}
\begin{split}
- p_{21,\xi}=&-c u_{21,\xi} + u_{11,\tau} + (u_{00}u_{11}+u_{10}u_{01})_\xi+ c^2z^2\eta_{00, \xi}\eta_{00,\xi\xi}\\
&+\frac{\Omega }{3}z(1-z^2)H_{1, \xi}-2\Omega c z\eta_{21,\xi} +2\Omega z \eta_{11,\tau}+2\Omega z H_{2, \xi}|_{z=1}.
\end{split}
\end{equation}
Combining \eqref{A-p-21-4} with \eqref{A-p-21-3}, we get
\begin{equation}\label{A-p-21-4a}
\begin{split}
&0=-c u_{21,\xi} + u_{11,\tau} + (u_{00}u_{11}+u_{10}u_{01})_\xi+ \bigg(\frac{c^2}{2}z^2+\frac{c(c  +4  c_1 )}{2}(z^2-1)\bigg)(\eta_{00, \xi}^2)_{\xi}\\
&+\frac{\Omega }{3}z(1-z^2)H_{1, \xi}+(1-2\Omega c)\eta_{21, \xi}+ 2\Omega \eta_{11,\tau}+\frac{z(z^2-1)}{6}H_{1, \xi}-\frac{c^2 }{2}(z^2-1) \eta_{10,\xi\xi\xi}\\
&+\bigg(c^2+\frac{c(3c+4c_1 )}{2} (z^2-1)\bigg)(\eta_{00}\eta_{00,\xi\xi})_{\xi}.
\end{split}
\end{equation}
Notice that
\begin{equation*}
\begin{split}
&(u_{01}u_{10}+u_{00}u_{11})_\xi \\
& = c^2(\eta_{01}\eta_{10}+\eta_{00}\eta_{11})_\xi + \left(\frac{c^2}{6}-\frac{2cc_1}{9} - \frac{c^2 z^2}{2}\right)(\eta_{00}\eta_{00,\xi\xi})_\xi- 3 c(c+c_1)(\eta_{00}^2\eta_{01})_\xi
\end{split}
\end{equation*}
and
\begin{equation*}
\begin{split}
&H_{2, \xi}|_{z=1} =3c^2(\eta_{01}\eta_{10}+\eta_{00}\eta_{11})_\xi - \left(\frac{c^2}{3}+\frac{2cc_1}{9}\right)(\eta_{00}\eta_{00,\xi\xi})_\xi- 3 c(c+c_1)(\eta_{00}^2\eta_{01})_\xi.
\end{split}
\end{equation*}
We substitute \eqref{A-u-21-xi-2} and \eqref{A-u-11-tau} into \eqref{A-p-21-4a} to get
\begin{equation}\label{A-eta-11-eqn}
\begin{split}
&2(\Omega + c)\eta_{11,\tau} + 3c^2(\eta_{00}\eta_{11}+\eta_{10}\eta_{01})_\xi-2(c+c_1)(3c+2c_1)(\eta_{00}^2\eta_{01})_\xi+\frac{c^2}{3}\eta_{10,\xi\xi\xi}\\
&-\left(\frac{c^2}{6}+\frac{10c c_1}{9}+\frac{2 c_1^2}{9}\right)(\eta_{00,\xi}^2)_{\xi}-\left(\frac{c^2}{3}+\frac{20 c c_1}{9}+\frac{8 c_1^2}{9}\right)(\eta_{00}\eta_{00,\xi\xi})_{\xi}=0.
\end{split}
\end{equation}

\vskip 0.2cm

\noindent {\bf Acknowledgments.}    The work of Gui is supported in part  by the NSF-China under the grants 11571279, 11331005, and the Foundation FANEDD-201315.  The work of Liu is  supported in part by the Simons Foundation grant-499875.

\vskip 0.2cm



\end{document}